\documentclass{amsart}

\usepackage{amssymb}

\usepackage{graphicx}

\usepackage{amsmath, amsfonts, mathtools, tikz}

\newtheorem{theorem}{Theorem}[section]
\newtheorem{lemma}[theorem]{Lemma}
\newtheorem{prop}[theorem]{Proposition}

\theoremstyle{definition}
\newtheorem{definition}[theorem]{Definition}

\theoremstyle{remark}

\theoremstyle{corollary}

\numberwithin{equation}{section}

\newcommand{\Bfamfive}{\{B^{5,l}\}_{l \geq 1}}
\newcommand{\veps}{\varepsilon}
\newcommand{\vphi}{\varphi}
\newcommand{\mZ}{\mathbb{Z}}

\allowdisplaybreaks
\def\ge{\mathfrak{g}}

\def\te{\mathfrak{t}}
\def\al{\alpha}
\def\cd{\cdots}
\def\cV{\mathcal{V}}
\def\cB{\mathcal{B}}
\def\C{\mathbb{C}}
\def\Z{\mathbb{Z}}
\def\la{\lambda}
\def\L{\Lambda}
\def\e{\tilde{e}}
\def\f{\tilde{f}}

\begin{document}

\title[$D_5^{(1)}$- Geometric Crystal and its ultra-discretization]{$D_5^{(1)}$- Geometric Crystal corresponding to the\\ Dynkin spin node $i=5$ and its ultra-discretization}


\author{Mana Igarashi}
\address{Department of Mathematics, 
Sophia University, Kioicho 7-1, Chiyoda-ku, Tokyo 102-8554,
Japan}
\email{mana-i@sophia.ac.jp}

\author{Kailash C. Misra}
\address{Department of Mathematics,
North Carolina State University, Raleigh, NC 27695-8205, USA}
\email{misra@ncsu.edu}
\thanks{KCM is partially supported by the Simons Foundation Grant \#307555.}

\author{Suchada Pongprasert}
\address{Department of Mathematics,
North Carolina State University, Raleigh, NC 27695-8205, USA}
\email{spathom@ncsu.edu }


\date{}

\begin{abstract}
Let $\ge$ be an affine Lie algebra with index set $I = \{0, 1, 2, \cdots , n\}$ and $\ge^L$ be its Langlands dual. It is conjectured that  for each Dynkin node $i \in I \setminus \{0\}$ the affine Lie algebra $\ge$ has a positive geometric crystal whose ultra-discretization is isomorphic to the limit of certain coherent family of perfect crystals for $\ge^L$. In this paper we construct a positive geometric crystal $\cV(D_5^{(1)})$ in the level zero fundamental spin $D_5^{(1)}$- module $W(\varpi_5)$. Then we define explicit $0$-action on the level $\ell$  known $D_5^{(1)}$- perfect crystal $B^{5, l}$ and show that $\{B^{5, l}\}_{l \geq 1}$ is a coherent family of perfect crystals with limit $B^{5, \infty}$. Finally we show that the ultra-discretization of  $\cV(D_5^{(1)})$ is isomorphic to $B^{5, \infty}$ as crystals which proves the conjecture in this case.
\end{abstract}

\maketitle

\section{Introduction}
Let $\ge$ be an affine Lie algebra \cite{Kac}  with Cartan datum $\{A, \Pi, \check{\Pi}, P, \check{P}\}$ and index set $I = \{0,1, \cdots , n\}$ where $A= (a_{ij})_{i,j \in I}$ is the affine GCM, $\Pi = \{\alpha_i \mid i \in I\}$ is the set of simple roots, $\check{\Pi} = \{\check{\alpha_i} \mid i \in I\}$ is the set of simple coroots, $P$ and $\check{P}$ are the weight lattice, and coweight lattice respectively. Let $\te = \C \otimes_{\mZ}\check{P}$, $\bf c$,  $\delta$, and  $\{\L_i \mid i\in I\}$ denote the Cartan subalgebra, the canonical central element, the null root and the set of fundamental weights respectively. Note that $\alpha_j(\check{\alpha}_i) = a_{ij}$ and $\L_j(\check{\alpha_i}) = \delta_{ij}$ and $\te = \text{span}_{\C}\{\check{\alpha}_i, d \mid i \in I\}$ where $d$ is a degree derivation. Then $P = \oplus_{j\in I}\mZ\L_j \oplus \mZ\delta \subset \te^*$, $\check{P} = \oplus_{i\in I}\mZ\check{\alpha}_i \oplus \mZ d \subset \te$ and $P_{cl} = P/\mZ\delta$ is called the classical weight lattice. The set $P^+ = \{\lambda \in P\mid \lambda (\check{\alpha}_i) \in \mZ_{\geq 0} \; \text{for} \; \text{all} \; i \in I\}$ (resp. $P_{cl}^+ = \{\lambda \in P_{cl} \mid \lambda (\check{\alpha}_i) \in \mZ_{\geq 0} \; \text{for} \; \text{all} \; i \in I\}$) is called the set of (affine) dominant (resp. classical dominant) weights and we say that $\lambda \in P^+$ or $P_{cl}^+$ has level $l = \lambda ({\bf c})$. We denote $(P^+)_l$ (resp. $(P_{cl}^+)_l$) to be the set of affine (resp. classical) dominant weights of level $l$. We denote $\te_{cl}^* = \te^*/{{\mathbb C}\delta}$ and 
$(\te_{cl}^*)_0 = \{\la \in \te_{cl}^* \mid \langle {\bf c}, \la\rangle = 0\}$. We also denote $\ge_i$ to be the subalgebra of $\ge$ with index set  $I_i = I \setminus \{i\}$ which is a finite dimensional semisimple  Lie algebra. The Weyl group $W$ of $\ge$ is generated by the simple reflections $\{s_i \mid i \in I\}$. The sets $\Delta$, $\Delta_+$ ,  $\Delta^{re}:=\{w(\al_i)|w\in W,\,\,i\in I\}$ and $\Delta^{re}_+ = \Delta_+ \cap \Delta^{re}$ are called the set of roots, postive roots, real roots and positive real roots respectively. In this paper we will assume $\ge$ to be simply laced which implies that the affine GCM $A = (a_{ij})_{i,j \in I}$ is symmetric.

Let $\ge'$ be the derived Lie algebra of $\ge$ and let  $G$ be the Kac-Moody group associated  with $\ge'$(\cite{KP, PK}).
Let $U_{\alpha}:=\exp\ge_{\alpha}$ $(\alpha\in \Delta^{re})$ be the one-parameter subgroup of $G$. The group $G$ is generated by 
$U_{\alpha}$ $(\alpha\in \Delta^{re})$. Let $U^{\pm}:=\langle U_{\pm\alpha}|\alpha\in\Delta^{re}_+\rangle$ be the subgroup generated by $U_{\pm\alpha}$
($\al\in \Delta^{re}_+$).
For any $i\in I$, there exists a unique homomorphism;
$\phi_i:SL_2(\C)\rightarrow G$ such that
\[
\hspace{-2pt}\phi_i\left(
\left(
\begin{array}{cc}
c&0\\
0&c^{-1}
\end{array}
\right)\right)=c^{\check{\alpha}_i},\,
\phi_i\left(
\left(
\begin{array}{cc}
1&t\\
0&1
\end{array}
\right)\right)=\exp(t e_i),\,
 \phi_i\left(
\left(
\begin{array}{cc}
1&0\\
t&1
\end{array}
\right)\right)=\exp(t f_i).
\]
where $c\in\C^\times$ and $t\in\C$.
Set $\check{\alpha}_i(c):=c^{\check{\alpha}_i}$,
$x_i(t):=\exp{(t e_i)}$, $y_i(t):=\exp{(t f_i)}$, 
$G_i:=\phi_i(SL_2(\C))$,
$H_i:=\phi_i(\{{\rm diag}(c,c^{-1})\mid 
c\in\C \setminus \{0\}\})$. 
Let $H$  be the subgroup of $G$ generated by $H_i$'s
with the Lie algebra $\te$. 
Then $H$ is called a  maximal torus in $G$, and 
$B^{\pm}=U^{\pm}H$ are the Borel subgroups of $G$.
The element $\bar{s}_i:=x_i(-1)y_i(1)x_i(-1) \in N_G(H)$  is a representative of 
$s_i\in W=N_G(H)/H$. 

The geometric crystal for the simply laced affine Lie algebra $\ge$ is defined as follows.
\begin{definition}\label{geometric}(\cite{BK},\cite{N}) 
The geometric crystal for the simply laced affine Lie algebra $\ge$ is a quadruple $\cV(\ge)=(X, \{e_i\}_{i \in I}, \{\gamma_i\}_{i \in I},$ 
$\{\veps_i\}_{i\in I})$, 
where $X$ is an ind-variety,  $e_i:\C^\times\times
X\longrightarrow X$ $((c,x)\mapsto e^c_i(x))$
are rational $\C^\times$-actions and  
$\gamma_i,\veps_i:X\longrightarrow 
\C$ $(i\in I)$ are rational functions satisfying the following:
\begin{enumerate}
\item $\{1\}\times X\subset {\rm dom}(e_i) \;
{\rm for} \; {\rm any} \; i\in I.$\\
\item $\gamma_j(e^c_i(x))=c^{a_{ij}}\gamma_j(x).$\\
\item $ \{e_i\}_{i\in I} \; {\rm satisfy \; the \; following \; relations}:\\
 \begin{array}{lll}
&\hspace{-20pt} \quad e^{c_1}_{i}e^{c_2}_{j}
=e^{c_2}_{j}e^{c_1}_{i}&
{\rm if }\,\,a_{ij}=a_{ji}=0,\\
&\hspace{-20pt} \quad e^{c_1}_{i}e^{c_1c_2}_{j}e^{c_2}_{i}
=e^{c_2}_{j}e^{c_1c_2}_{i}e^{c_1}_{j}&
{\rm if }\,\,a_{ij}=a_{ji}=-1,\\
\end{array}$\\
\item $\veps_i(e_i^c(x))=c^{-1}\veps_i(x)$ and $\veps_i(e_j^c(x))=\veps_i(x) \qquad {\rm if }\,
a_{i,j}=a_{j,i}=0.$
\end{enumerate}
\end{definition}
For fixed $i \in I$, let $G^i$ be the reductive algebraic group with Lie algebra $\ge_i$ and $B^i$, $W^i$ be its Borel subgroup, Weyl group respectively.
We consider the flag variety $X^i:=G^i/{B^i}$. For $w \in W^i$, the Schubert cell $X^i_w$ associated with $w$ has a natural $\ge_i$-geometric crystal structure \cite{BK, N}. Let $w=s_{i_1} s_{i_2} \cdots s_{i_l}$ be a reduced expression. For ${\bf i}:=(i_1, i_2, \cdots ,i_l)$, set 
\begin{equation*}
B_{\bf i}^-
:=\{Y_{\bf i}(c_1, c_2, \cdots ,c_l)
:=Y_{i_1}(c_1) Y_{i_2}(c_2)\cdots Y_{i_l}(c_l)
\,\vert\, c_1, c_2, \cdots ,c_l\in\C^\times\}\subset B^-,
\label{bw1}
\end{equation*}
where $Y_j(c):=y_j(\frac{1}{c}){\check\al}_j(c) = y_j(\frac{1}{c})c^{{\check\al}_j}$. Then  we have the following result.
\begin{theorem}\label{schubert}\cite{BK, N}
The set $B_{\bf i}^-$ with the explicit actions of \; $e^c_k$, $\veps_k$, and $\gamma_k$ , for $k \in I_i, c \in \C^\times$ given by:
\begin{eqnarray}
&& e_k^c(Y_{\bf i}(c_1,\cdots,c_l))
=Y_{\bf i}({\mathcal C}_1,\cdots,{\mathcal C}_l)), \notag\\
&&\text{where} \notag\\
&&{\mathcal C}_j:=
c_j\cdot \frac{\displaystyle \sum_{1\leq m\leq j,i_m=k}
 \frac{c}
{c_1^{a_{i_1,k}}\cdots c_{m-1}^{a_{i_{m-1},k}}c_m}
+\sum_{j< m\leq k,i_m=k} \frac{1}
{c_1^{a_{i_1,k}}\cdots c_{m-1}^{a_{i_{m-1},k}}c_m}}
{\displaystyle\sum_{1\leq m<j,i_m=k} 
 \frac{c}{c_1^{a_{i_1,k}}\cdots c_{m-1}^{a_{i_{m-1},k}}c_m}+ 
\mathop\sum_{j\leq m\leq k,i_m=k}  \frac{1}
{c_1^{a_{i_1,k}}\cdots c_{m-1}^{a_{i_{m-1},k}}c_m}},\\
&& \veps_k(Y_{\bf i}(c_1,\cd,c_l))=
\sum_{1\leq m\leq l,i_m=k} \frac{1}
{c_1^{a_{i_1,k}}\cdots c_{m-1}^{a_{i_{m-1},k}}c_m},\\
&&\gamma_k(Y_{\bf i}(c_1,\cdots,c_l))
=c_1^{a_{i_1,k}}\cdots c_l^{a_{i_l,k}},
\end{eqnarray}
is a geometric crystal isomorphic to $X^i_w$.
\end{theorem}

The geometric crystal $\cV(\ge)=(X, \{e_i\}_{i \in I}, 
\{\gamma_i\}_{i \in I}, \{\veps_i\}_{i\in I})$ is said to be positive if it has a 
positive structure \cite{BK, N, KNO}. 
Roughly speaking this means that each of the rational maps 
$e^c_i$, $\veps_i$  and $\gamma_i$ are
given by ratio of polynomial functions with positive coefficients. 
For example, $B_{\bf i}^-$ is a positive geometric crystal.

For a dominant weight  $\lambda \in P^+$ of level $l$,
Kashiwara defined the crystal base $(L(\lambda), B(\lambda))$
\cite{Kas1} (also see \cite{Lu})  for the integrable highest weight $\ge$-module $V(\lambda)$. As shown in \cite{KMN1}, the crystal
$B(\lambda)$ can be realized as a set of paths in the semi-infinite tensor product $ \cdots \otimes B^l \otimes B^l \otimes B^l$ where
$B^l$ is a perfect crystal of level $l$. This is called the path realization of the crystal $B(\lambda)$. When a family of perfect crystals 
$\{B^l\}_{l\geq 1}$ are coherent \cite{KKM}, it has a limit $B^\infty$. The positive geometric crystals are related to Kashiwara crystals via the 
ultra-discretization functor $\mathcal{UD}$ \cite{BK, N} which transforms positive rational 
functions to piecewise-linear functions by the simple correspondence:
$$
x \times y \longmapsto x+y, \qquad \frac{x}{y} \longmapsto x - y, 
\qquad x + y \longmapsto {\rm max}\{x, y\}.
$$
It was conjectured in \cite{KNO} that for each affine Lie algebra $\ge$ and 
each Dynkin index $k \in I \setminus \{0\}$, there exists a positive geometric crystal
$\cV(\ge)=(X, \{e_i\}_{i \in I}, \{\gamma_i\}_{i \in I}, \\
\{\veps_i\}_{i\in I})$ whose ultra-discretization $\mathcal{UD}(\cV)$ is isomorphic 
to the limit $B^{\infty}$ of a coherent family of perfect crystals for the Langlands dual $\ge^L$. If $\ge$ is simply laced then the Langland dual is 
$\ge$ itself. So far this conjecture has been proved for the Dynkin index $k = 1$ and $\ge = A_n^{(1)}, 
B_n^{(1)}, C_n^{(1)}, D_n^{(1)}, A_{2n-1}^{(2)}, A_{2n}^{(2)},$
$D_{n+1}^{(2)}$ \cite{KNO}, $\ge = D_4^{(3)}$ \cite{IMN}, $\ge = G_2^{(1)}$ \cite{N4}. For $k > 1$ , this conjecture has been shown to hold only for $\ge = A_n^{(1)}$ \cite{MN1, MN2}.

In this paper we prove the conjecture in \cite{KNO} for $\ge = D_5^{(1)}$ and Dynkin idex $k = 5$, the spin node. In Section 2,  we construct a positive geometric crystal $\cV(D_5^{(1)})$ in the level zero fundamental spin module $W(\varpi_5)$. In Section 3, for $l \geq 1$ we coordinatize the perfect crystal $B^{5,l}$ for $D_5^{(1)}$ given in \cite{KMN2} and define an explicit $0$-action. Then we show that the family of perfect crystals $\{B^{5,l}\}_{l\geq 1}$ is a coherent family and determine its limit $B^{5,\infty}$. In the last section we ultra-discretize the positive geometric crystal $\cV(D_5^{(1)})$ and show that it is isomorphic to  $B^{5,\infty}$ as crystals.

\section{Affine Geometric Crystal \bf{$\mathcal{V}(D_5^{(1)})$}}

From now on we assume $\ge$ to be the affine Lie algebra $D_5^{(1)}$ with index set $I = \{0,1,2,3,4,5\}$, Cartan matrix $A = (a_{ij})_{i,j \in I}$ where $a_{ii} = 2, a_{j,j + 1} = -1 = a_{j+1,j}, \; j = 1,2,3, a_{02} = a_{20} = a_{35} = a_{53} = -1, a_{ij} = 0$ otherwise and Dynkin diagram:
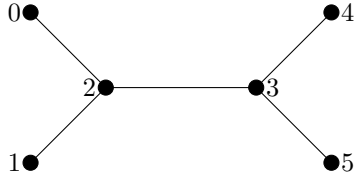
\begin{center}
\begin{tikzpicture}
\draw (-2,1)--(-1,0); \draw (-2,-1)--(-1,0); \draw (-1,0)--(1,0); \draw (1,0)--(2,1); \draw (1,0)--(2,-1);
\draw [fill] (-2,1) circle [radius=0.1] node[left=.1pt] (a) {0};
\draw [fill] (-2,-1) circle [radius=0.1] node[left=.1pt] (b) {1};
\draw [fill] (-1,0) circle [radius=0.1] node[left=.2pt] (c) {2};
\draw [fill] (1,0) circle [radius=0.1] node[right=.2pt] (d) {3};
\draw [fill] (2,1) circle [radius=0.1] node[right=.1pt] (e) {4};
\draw [fill] (2,-1) circle [radius=0.1] node[right=.1pt] (f) {5};        
\end{tikzpicture}
\end{center}
Let $\{\alpha_0, \alpha_1, \alpha_2, \alpha_3, \alpha_4, \alpha_5\}, \ \{\check{\alpha_0}, \check{\alpha_1}, \check{\alpha_2}, \check{\alpha_3}, \check{\alpha_4}, \check{\alpha_5}\}$ and $\{\Lambda_0, \Lambda_1, \Lambda_2, \Lambda_3, \Lambda_4, \Lambda_5\}$ denote the set of simple roots, simple coroots and fundamental weights, respectively.
Then ${\bf c} =\check{\alpha_0}+\check{\alpha_1}+2\check{\alpha_2}+2\check{\alpha_3}+\check{\alpha_4}+\check{\alpha_5}$ and $\delta = \al_0 +\al_1+2\al_2+2\al_3+\al_4+\al_5$ are the central element and null root respectively. The sets $P_{cl} = \oplus_{j=0}^5 \Z\L_j$ and $P = P_{cl}\oplus\Z\delta$ are called classical weight lattice and weight lattice respectively.

We consider the Dynkin diagram automorphism\footnote{We thank T. Nakashima for suggesting it to us.} $\sigma$ defined by 
$$\sigma : 0 \mapsto 5, 1 \mapsto 4, 2 \mapsto 3, 3 \mapsto 2, 4 \mapsto 0, 5 \mapsto 1.$$

\medskip

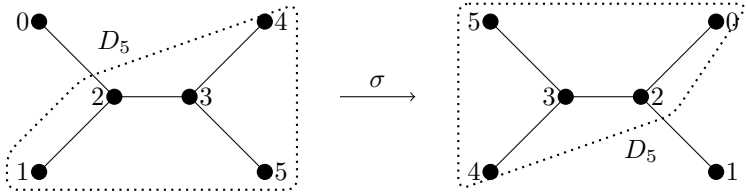
\begin{tikzpicture}
\draw (-10,4)--(-9,3); \draw (-10,2)--(-9,3); \draw (-9,3)--(-8,3); \draw (-8,3)--(-7,4); \draw (-8,3)--(-7,2);  \draw (-1,4)--(-2,3); \draw (-1,2)--(-2,3); \draw (-2,3)--(-3,3); \draw (-3,3)--(-4,4); \draw (-3,3)--(-4,2);
\draw [fill] (-10,4) circle [radius=0.1] node[left=.1pt] (a) {0};
\draw [fill] (-10,2) circle [radius=0.1] node[left=.1pt] (b) {1};
\draw [fill] (-9,3) circle [radius=0.1] node[left=.2pt] (c) {2};
\draw [fill] (-8,3) circle [radius=0.1] node[right=.2pt] (d) {3};
\draw [fill] (-7,4) circle [radius=0.1] node[right=.1pt] (e) {4};
\draw [fill] (-7,2) circle [radius=0.1] node[right=.1pt] (f) {5};
\draw [->](-6,3)--(-5,3) ;
\draw [fill=white] (-5.5,3) node[above=.1pt, black] {$\sigma$};
\draw [fill] (-1,4) circle [radius=0.1] node[right=.1pt] (g) {0};
\draw [fill] (-1,2) circle [radius=0.1] node[right=.1pt] (h) {1};
\draw [fill] (-2,3) circle [radius=0.1] node[right=.2pt] (i) {2};
\draw [fill] (-3,3) circle [radius=0.1] node[left=.2pt] (j) {3};
\draw [fill] (-4,4) circle [radius=0.1] node[left=.1pt] (k) {5};
\draw [fill] (-4,2) circle [radius=0.1] node[left=.1pt] (l) {4};
\path [draw,black, dotted,thick,rounded corners] 
               (b.north west) 
            -- (c.north west) 
            -- (e.north east) 
            -- (f.south east) 
            -- (b.south west) 
            -- cycle ;   
\draw [fill=white] (-9,4) node[below=.1pt, black] {$D_5$};             

\path [draw,black, dotted,thick, rounded corners] 
               (k.north west) 
            -- (g.north east) 
            -- (i.south east) 
            -- (l.south west) 
            -- (k.south west) 
            -- cycle ;   

\draw [fill=white] (-2,2) node[above=.1pt, black] {$D_5$};  
\end{tikzpicture}

\medskip

Let $\ge_j$ (resp. $\sigma(\ge)_j)$) be the subalgebra of $\ge$ 
(resp. $\sigma(\ge)$) with index set $I_j= I \setminus \{j\}$. Then observe that $\ge_0$ as well as $\ge_1$ and $\sigma(\ge)_1$ are isomorphic to $D_5$.

Let $W(\varpi_5)$ be the level $0$ fundamental $U'_q(\mathfrak{g})$-module associated with the level $0$ weight $\varpi_5 = \L_5 - \L_0$  \cite{K0}. By [\cite{K0}, Theorem 5.17], $W(\varpi_5)$ is a finite-dimensional irreducible integrable $U'_q(\mathfrak{g})$-module and has a global basis with a simple crystal. Thus, we can consider the specialization $q=1$ and obtain the finite-dimensional $D_5^{(1)}$-module $W(\varpi_5)$, which we call the fundamental $D_5^{(1)}$- module  and use the same notation as above. Below we give the explicit description of $W(\varpi_5)$.
\subsection{Fundamental Representation \bf{$W(\varpi_5)$} for \bf{$D_5^{(1)}$}} 
The fundamental $D_5^{(1)}$-module $W(\varpi_5)$ is a 16-dimensional module with the basis 
$$\{v= (i_1, i_2, i_3, i_4, i_5) |i_j \in \{+, -\}, \ i_1 i_2 i_3 i_4 i_5 = +\}.$$
The actions of the generators $e_k$, and $f_k$, $0\le k \le 5$ of $D_5^{(1)}$, on the basis vectors are given as follows.
\begin{displaymath}
\begin{split} 
& e_k (i_1, i_2, i_3, i_4, i_5) = \left \{
	\begin{array}{lllll}
	(-, -, i_3, i_4, i_5)    \hspace{.6cm}& & \text{ if  }  k = 0, \ (i_1, i_2)=(+,+)\\
	(i_1, \ldots, +, -,\ldots, i_5) & &\text{ if  } k \neq 0, \ k \neq 5, \ (i_k, i_{k+1})=(-,+)\\
	\hspace{1.2cm}k \hspace{.28cm}k+1 \\
	(i_1, i_2, i_3, +, +)    \hspace{.6cm}& & \text{ if  } k = 5, \ (i_4, i_5)=(-,-)\\
	0 \hspace{1cm} & &\text{ otherwise.}
	\end{array}
	\right. \\
& f_k (i_1, i_2, i_3, i_4, i_5) = \left \{
	\begin{array}{lllll}
	(+, +, i_3, i_4, i_5)    \hspace{.6cm}& & \text{ if  }  k = 0, \ (i_1, i_2)=(-,-)\\
	(i_1, \ldots, -, +,\ldots, i_5) & &\text{ if  } k \neq 0, \ k \neq 5, \ (i_k, i_{k+1})=(+,-)\\
	\hspace{1.2cm}k \hspace{.28cm}k+1 \\
	(i_1, i_2, i_3, -, -)    \hspace{.6cm}& & \text{ if  } k = 5, \ (i_4, i_5)=(+,+)\\
	0 \hspace{1cm} & &\text{ otherwise.}
	\end{array}
	\right.
\end{split}
\end{displaymath}

Furthermore, we observe that
\begin{displaymath}
\begin{split} 
& \langle {\check\alpha_k}, {\rm wt}(v)\rangle = \left \{
\begin{array}{lllll}
1   \hspace{.1cm}& & \text{ if  } \ k = 0, \ (i_1, i_2)=(-,-)\\
& & \text{ or  }  k \neq 0, \ k \neq 5, \ (i_k, i_{k+1})=(+,-)\\
& & \text{ or  }  k = 5, (i_4, i_5)=(+,+)\\
-1 & & \text{ if  } \ k = 0, \ (i_1, i_2)=(+,+)\\
& & \text{ or  }  k \neq 0, \ k \neq 5, (i_k, i_{k+1})=(-,+)\\
& & \text{ or  }  k = 5, \ (i_4, i_5)=(-,-)\\
0 & & \text{ otherwise.}
\end{array}
\right. \\
\end{split}
\end{displaymath}

Note that in $W(\varpi_5)$, we have $(+,+,+,+,+)$  (resp. $(-,+,+,+,-)$) is a $\mathfrak{g}_0$ (resp. $\mathfrak{g}_1$) highest weight vector with weight $\varpi_5 = \Lambda_5-\Lambda_0$ (resp. $\check{\varpi_5}:=\Lambda_4-\Lambda_1$). We define $\sigma(\L_j) = \L_{\sigma(j)}$ for $j \in I$. Then we define the action of $\sigma$ on $W(\varpi_5)$ by $\sigma(v) = v'$ if $\sigma({\rm wt}(v)) = {\rm wt}(v')$. 

\subsection{Affine Geometric Crystal \bf{$\mathcal{V}(D_5^{(1)})$} in \bf{$W(\varpi_5)$}}
Now we will construct the affine geometric crystal $\mathcal{V}(D_5^{(1)})$ in $W(\varpi_5)$ explicitly. For $\xi \in (\mathfrak{t}^*_{\text{cl}})_0$, let $t(\xi)$ be the translation as in [\cite{K0}, Sect 4]. Define simple reflections $s_k (\lambda) := \lambda - \lambda({\check\alpha_k})\alpha_k , k \in I$ and let $W = \langle s_k \mid k \in I \rangle$ be the Weyl group for $D_5^{(1)}$. Then we have
\begin{align*}
&t(\varpi_5) = \sigma s_4s_3s_2s_5s_3s_4s_1s_2s_3s_5=\sigma w_1,\\
&t(\check{\varpi_5}) = \sigma s_5s_3s_2s_4s_3s_5s_0s_2s_3s_4=\sigma w_2,
\end{align*}
where $w_1 = s_4s_3s_2s_5s_3s_4s_1s_2s_3s_5 \in W^0$ and $w_2 = s_5s_3s_2s_4s_3s_5s_0s_2s_3s_4 \in W^1$.
Associated with these Weyl group elements $w_1, w_2 \in W$, we define algebraic varieties $\mathcal{V}_1, \mathcal{V}_2 \subset W(\varpi_5)$ as follows.
\begin{align*}
\mathcal{V}_1	&=\big\{V_1(x):=Y_4(x_4^{(2)})Y_3(x_3^{(3)})Y_2(x_2^{(2)})Y_5(x_5^{(2)})Y_3(x_3^{(2)})Y_4(x_4^{(1)})\\
&\hspace{1cm}Y_1(x_1^{(1)})Y_2(x_2^{(1)})Y_3(x_3^{(1)})Y_5(x_5^{(1)})(+,+,+,+,+) | x_m^{(l)} \in \mathbb{C}^{\times} \big\}, \\
\mathcal{V}_2	&=\big\{V_2(y):=Y_5(y_5^{(2)})Y_3(y_3^{(3)})Y_2(y_2^{(2)})Y_4(y_4^{(2)})Y_3(y_3^{(2)})Y_5(y_5^{(1)})\\
&\hspace{1cm}Y_0(y_0^{(1)})Y_2(y_2^{(1)})Y_3(y_3^{(1)})Y_4(y_4^{(1)})(-,+,+,+,-) | y_m^{(l)} \in \mathbb{C}^{\times} \big\}, 
\end{align*}
where $x=(x_4^{(2)},x_3^{(3)}, x_2^{(2)}, x_5^{(2)}, x_3^{(2)}, x_4^{(1)}, x_1^{(1)}, x_2^{(1)}, x_3^{(1)}, x_5^{(1)})$ and $y=(y_5^{(2)}, y_3^{(3)}, y_2^{(2)}, \\ y_4^{(2)}, y_3^{(2)}, y_5^{(1)}, y_0^{(1)}, y_2^{(1)}, y_3^{(1)}, y_4^{(1)})$.

From the explicit actions of $f_k$'s on $W(\varpi_5)$, we observe that $f_k^2 =0$, for all $k \in I$. Therefore, we have
$$Y_k(c)=(1+\frac{f_k}{c}){\check\alpha_k}(c) = (1+\frac{f_k}{c})c^{\check\alpha_k} \ \text{for all} \ k \in I.$$ 
Thus we have the explicit forms of $V_1(x)$ and $V_2(y)$ as follows. \\

$V_1 (x) = x_5^{(2)}x_5^{(1)}(+,+,+,+,+) + \big(x_3^{(3)}x_5^{(1)} + \frac{x_3^{(3)}x_3^{(2)}x_3^{(1)}}{x_5^{(2)}}\big)(+,+,+,-,-) + \big(x_4^{(2)}x_5^{(1)} + \frac{x_4^{(2)}x_3^{(2)}x_3^{(1)}}{x_5^{(2)}} + \frac{x_4^{(2)}x_2^{(2)}x_3^{(1)}}{x_3^{(3)}} + \frac{x_4^{(2)}x_2^{(2)}x_4^{(1)}x_2^{(1)}}{x_3^{(3)}x_3^{(2)}}\big)(+,+,-,+,-) + \big(x_5^{(1)} + \frac{x_3^{(2)}x_3^{(1)}}{x_5^{(2)}} + \frac{x_2^{(2)}x_3^{(1)}}{x_3^{(3)}} + \frac{x_2^{(2)}x_4^{(1)}x_2^{(1)}}{x_3^{(3)}x_3^{(2)}} + \frac{x_2^{(2)}x_2^{(1)}}{x_4^{(2)}} \big)(+,+,-,-,+) + \big(x_4^{(2)}x_3^{(1)} + \frac{x_4^{(2)}x_4^{(1)}x_2^{(1)}}{x_3^{(2)}} + \frac{x_4^{(2)}x_4^{(1)}x_1^{(1)}}{x_2^{(2)}} \big)(+,-,+,+,-) + \big(x_3^{(1)} + \frac{x_4^{(1)}x_2^{(1)}}{x_3^{(2)}} + \frac{x_4^{(1)}x_1^{(1)}}{x_2^{(2)}} + \frac{x_3^{(3)}x_2^{(1)}}{x_4^{(2)}} + \frac{x_3^{(3)}x_3^{(2)}x_1^{(1)}}{x_4^{(2)}x_2^{(2)}} \big)(+,-,+,-,+) + \big(x_2^{(1)} + \frac{x_3^{(2)}x_1^{(1)}}{x_2^{(2)}} + \frac{x_5^{(2)}x_1^{(1)}}{x_3^{(3)}} \big)(+,-,-,+,+) + x_1^{(1)}(+,-,-,-,-) + x_4^{(2)}x_4^{(1)}(-,+,+,+,-) + \\ \big(x_4^{(1)} + \frac{x_3^{(3)}x_3^{(2)}}{x_4^{(2)}}\big)(-,+,+,-,+) + \big(x_3^{(2)} + \frac{x_2^{(2)}x_5^{(2)}}{x_3^{(3)}}\big)(-,+,-,+,+) +  x_5^{(2)}(-,-,+,+,+) +  x_2^{(2)}(-,+,-,-,-) + x_3^{(3)}(-,-,+,-,-) +  x_4^{(2)}(-,-,-,+,-) + (-,-,-,-,+)$,\\

$V_2 (y) = y_4^{(2)}y_4^{(1)}(-,+,+,+,-) + \big(y_3^{(3)}y_4^{(1)} + \frac{y_3^{(3)}y_3^{(2)}y_3^{(1)}}{y_4^{(2)}}\big)(-,+,+,-,+) + \big(y_5^{(2)}y_4^{(1)} + \frac{y_5^{(2)}y_3^{(2)}y_3^{(1)}}{y_4^{(2)}} + \frac{y_5^{(2)}y_2^{(2)}y_3^{(1)}}{y_3^{(3)}} + \frac{y_5^{(2)}y_2^{(2)}y_5^{(1)}y_2^{(1)}}{y_3^{(3)}y_3^{(2)}}\big)(-,+,-,+,+) + \big(y_4^{(1)} + \frac{y_3^{(2)}y_3^{(1)}}{y_4^{(2)}} + \frac{y_2^{(2)}y_3^{(1)}}{y_3^{(3)}} + \frac{y_2^{(2)}y_5^{(1)}y_2^{(1)}}{y_3^{(3)}y_3^{(2)}} + \frac{y_2^{(2)}y_2^{(1)}}{y_5^{(2)}} \big)(-,+,-,-,-) + \big(y_5^{(2)}y_3^{(1)} + \frac{y_5^{(2)}y_5^{(1)}y_2^{(1)}}{y_3^{(2)}} + \frac{y_5^{(2)}y_5^{(1)}y_0^{(1)}}{y_2^{(2)}} \big)(-,-,+,+,+) + \big(y_3^{(1)} + \frac{y_5^{(1)}y_2^{(1)}}{y_3^{(2)}} + \frac{y_5^{(1)}y_0^{(1)}}{y_2^{(2)}} + \frac{y_3^{(3)}y_2^{(1)}}{y_5^{(2)}} + \frac{y_3^{(3)}y_3^{(2)}y_0^{(1)}}{y_5^{(2)}y_2^{(2)}} \big)(-,-,+,-,-) + \big(y_2^{(1)} + \frac{y_3^{(2)}y_0^{(1)}}{y_2^{(2)}} + \frac{y_4^{(2)}y_0^{(1)}}{y_3^{(3)}} \big)(-,-,-,+,-) + y_0^{(1)}(-,-,-,-,+) + y_5^{(2)}y_5^{(1)}(+,+,+,+,+) +\\ \big(y_5^{(1)} + \frac{y_3^{(3)}y_3^{(2)}}{y_5^{(2)}}\big)(+,+,+,-,-) + \big(y_3^{(2)} + \frac{y_2^{(2)}y_4^{(2)}}{y_3^{(3)}}\big)(+,+,-,+,-) + y_4^{(2)}(+,-,+,+,-) +  y_2^{(2)}(+,+,-,-,+) + y_3^{(3)}(+,-,+,-,+) +  y_5^{(2)}(+,-,-,+,+) + (+,-,-,-,-)$. \\

Now for a given $x$ we solve the equation
\begin{align}
V_2(y) = a(x)\sigma (V_1(x)) \label{ytox}.
\end{align}
where $a(x)$ is a rational function in $x$ and the action of $\sigma$ on $V_1(x)$ is induced by its action on $W(\varpi_5)$. Though this equation is over-determined, it can be solved uniquely by comparing the coefficients of the basis vectors of $W(\varpi_5)$. We give the explicit  solutions of $a(x)$, and the variables $y_m^{(l)}$ below.
\begin{lemma} \label{yinx}
The rational function $a(x)$ and the complete solution of (\ref{ytox}) is:
\begin{align*} 
a(x) &=\frac{1}{x_5^{(2)}x_5^{(1)}}, \\
y_5^{(2)} &=\frac{1}{x_5^{(1)}} , \qquad y_0^{(1)} =\frac{x_4^{(2)}x_4^{(1)}}{x_5^{(2)}x_5^{(1)}},
		&y_5^{(1)} &=\frac{1}{x_5^{(2)}},\\
y_4^{(2)} &=\frac{x_1^{(1)}}{x_5^{(2)}x_5^{(1)}} \Big( \frac{x_5^{(2)}}{x_3^{(3)}}+\frac{x_3^{(2)}}{x_2^{(2)}}+\frac{x_2^{(1)}}{x_1^{(1)}}\Big), \hspace{-.5cm} 
		& y_4^{(1)} &=\Big( \frac{x_5^{(2)}}{x_3^{(3)}}+\frac{x_3^{(2)}}{x_2^{(2)}}+\frac{x_2^{(1)}}{x_1^{(1)}}\Big)^{-1},
\end{align*}
\begin{align*}
\hspace{1cm}y_3^{(3)} &=\frac{1}{x_5^{(1)}} \Big( \frac{x_2^{(2)}}{x_3^{(3)}}+\frac{x_3^{(2)}}{x_5^{(2)}}\Big),\\
y_3^{(2)} &=\frac{1}{x_5^{(2)}x_5^{(1)}} \Big( \frac{x_2^{(2)}}{x_3^{(3)}}+\frac{x_3^{(2)}}{x_5^{(2)}}\Big)^{-1} \Big( \frac{x_2^{(2)}x_2^{(1)}}{x_4^{(2)}}+\frac{x_2^{(2)}x_3^{(1)}}{x_3^{(3)}} +\frac{x_3^{(2)}x_3^{(1)}}{x_5^{(2)}} + \frac{x_2^{(2)}x_4^{(1)}x_2^{(1)}}{x_3^{(3)}x_3^{(2)}}\Big),\\
y_3^{(1)} &=\frac{x_2^{(2)}x_2^{(1)}}{x_5^{(2)}} \Big( \frac{x_2^{(2)}x_2^{(1)}}{x_4^{(2)}}+\frac{x_2^{(2)}x_3^{(1)}}{x_3^{(3)}} +\frac{x_3^{(2)}x_3^{(1)}}{x_5^{(2)}} +\frac{x_2^{(2)}x_4^{(1)}x_2^{(1)}}{x_3^{(3)}x_3^{(2)}}\Big)^{-1},
\end{align*}
\begin{align*} 
\hspace{1cm}y_2^{(2)} &=\frac{1}{x_5^{(1)}} \Big( \frac{x_3^{(3)}x_3^{(2)}}{x_4^{(2)}x_5^{(2)}}+\frac{x_4^{(1)}}{x_5^{(2)}}\Big), \hspace{1cm} 
	& y_2^{(1)} &=\frac{x_3^{(3)}x_3^{(2)}x_3^{(1)}}{(x_5^{(2)})^2x_5^{(1)}} \Big( \frac{x_3^{(3)}x_3^{(2)}}{x_4^{(2)}x_5^{(2)}}+\frac{x_4^{(1)}}{x_5^{(2)}}\Big)^{-1}.\\
\end{align*}
\end{lemma}

Using Lemma \ref{yinx} we define the map
$$\bar{\sigma}: \mathcal{V}_1 \rightarrow \mathcal{V}_2,$$
$$V_1(x) \mapsto V_2(y).$$
Now we have the following result.

\begin{prop}\label{mapsigmabar} 
The map $\bar{\sigma}: \mathcal{V}_1 \rightarrow \mathcal{V}_2$ is a bi-positive birational isomorphism with the inverse positive rational map 
$$\bar{\sigma}^{-1}: \mathcal{V}_2 \rightarrow \mathcal{V}_1,$$
$$V_2(y) \mapsto V_1(x)$$
given by
\begin{align*} 
\hspace{.1cm} x_5^{(2)} &=\frac{1}{y_5^{(1)}},
		&x_5^{(1)} &=\frac{1}{y_5^{(2)}},\\
x_4^{(2)} &=\frac{y_0^{(1)}}{y_2^{(2)}}+\frac{y_2^{(1)}}{y_3^{(2)}}+\frac{y_3^{(1)}}{y_5^{(1)}}, \hspace{0cm} 
		& x_4^{(1)} &=\frac{y_0^{(1)}}{y_5^{(2)}y_5^{(1)}} \Big( \frac{y_0^{(1)}}{y_2^{(2)}}+\frac{y_2^{(1)}}{y_3^{(2)}}+\frac{y_3^{(1)}}{y_5^{(1)}}\Big)^{-1},
\end{align*}
\begin{align*}
x_3^{(3)} &=\frac{y_2^{(2)}y_2^{(1)}}{y_3^{(3)}y_3^{(2)}} + \frac{y_2^{(2)}y_3^{(1)}}{y_3^{(3)}y_5^{(1)}} + \frac{y_3^{(2)}y_3^{(1)}}{y_4^{(2)}y_5^{(1)}} + \frac{y_4^{(1)}}{y_5^{(1)}},\\
x_3^{(2)} &=\frac{y_2^{(2)}}{y_5^{(2)}y_5^{(1)}} \Big( \frac{y_2^{(2)}y_2^{(1)}}{y_3^{(3)}y_3^{(2)}} + \frac{y_2^{(2)}y_3^{(1)}}{y_3^{(3)}y_5^{(1)}} + \frac{y_3^{(2)}y_3^{(1)}}{y_4^{(2)}y_5^{(1)}} + \frac{y_4^{(1)}}{y_5^{(1)}} \Big)^{-1} \Big( \frac{y_2^{(1)}}{y_3^{(2)}}+\frac{y_3^{(1)}}{y_5^{(1)}}\Big),\\
x_3^{(1)} &=\frac{y_2^{(1)}}{y_5^{(2)}y_5^{(1)}}  \Big( \frac{y_2^{(1)}}{y_3^{(2)}}+\frac{y_3^{(1)}}{y_5^{(1)}}\Big)^{-1}, \qquad x_1^{(1)} =\frac{y_4^{(2)}y_4^{(1)}}{y_5^{(2)}y_5^{(1)}},
\end{align*}
\begin{align*} 
\hspace{.1cm} x_2^{(2)} &=\frac{y_3^{(3)}}{y_5^{(2)}} \Big(\frac{y_3^{(2)}y_3^{(1)}}{y_4^{(2)}y_5^{(1)}}+\frac{y_4^{(1)}}{y_5^{(1)}} \Big), \hspace{-.5cm} 
	& x_2^{(1)} &=\frac{y_3^{(2)}y_3^{(1)}}{y_5^{(2)}(y_5^{(1)})^2} \Big( \frac{y_3^{(2)}y_3^{(1)}}{y_4^{(2)}y_5^{(1)}}+\frac{y_4^{(1)}}{y_5^{(1)}}\Big)^{-1}.\\
\end{align*}
\begin{proof} The fact that $\bar{\sigma}$ is a bi-positive birational map follows from the explicit formulas. The rest follows by direct calculations.
\end{proof}
\end{prop}

It is known that $\cV_1$ (resp. $\cV_2$) has the structure of a $\ge_0$ (resp. $\ge_1$) positive geometric crystal (\cite{BK}, \cite{N}, \cite{KNO}). Indeed, note that taking the sesquence  ${\bf i} = (4,3,2,5,3,4,1,2,3,5)$ the explicit actions of \ $e_k^c, \ \gamma_k, \ \veps_k$ on $V_1(x)$ for $k=1,2,3,4,5$ are given by Theorem \ref{schubert}  as follows.
\begin{align*}
e_k^c(V_1(x)) &=
\begin{cases}
V_1(x_4^{(2)},x_3^{(3)},x_2^{(2)},x_5^{(2)},x_3^{(2)},x_4^{(1)},cx_1^{(1)},x_2^{(1)},x_3^{(1)},x_5^{(1)}), &k=1,\\
V_1(x_4^{(2)},x_3^{(3)},c_2 x_2^{(2)} ,x_5^{(2)},x_3^{(2)},x_4^{(1)},x_1^{(1)},\frac{c}{c_2}x_2^{(1)},x_3^{(1)},x_5^{(1)}), &k=2,\\
V_1(x_4^{(2)}, c_{3_1} x_3^{(3)},x_2^{(2)},x_5^{(2)}, c_{3_2} x_3^{(2)},x_4^{(1)},x_1^{(1)},x_2^{(1)}, \frac{c}{c_{3_1} c_{3_2}}x_3^{(1)},x_5^{(1)}), &k=3,\\
V_1(c_4 x_4^{(2)},x_3^{(3)} ,x_2^{(2)},x_5^{(2)},x_3^{(2)},\frac{c}{c_4} x_4^{(1)},x_1^{(1)},x_2^{(1)},x_3^{(1)},x_5^{(1)}), &k=4,\\
V_1(x_4^{(2)},x_3^{(3)},x_2^{(2)},c_5 x_5^{(2)},x_3^{(2)},x_4^{(1)},x_1^{(1)},x_2^{(1)},x_3^{(1)}, \frac{c}{c_5} x_5^{(1)} ), &k=5,
\end{cases} \\
\text{where} & \\
&c_2 =\frac{cx_2^{(2)}x_2^{(1)}+x_3^{(2)}x_1^{(1)}}{x_2^{(2)}x_2^{(1)}+x_3^{(2)}x_1^{(1)}}, \\
&c_{3_1} = \frac{cx_3^{(3)}(x_3^{(2)})^2 x_3^{(1)} + x_2^{(2)}x_5^{(2)}x_3^{(2)}x_3^{(1)} + x_2^{(2)}x_5^{(2)}x_4^{(1)}x_2^{(1)}}{x_3^{(3)}(x_3^{(2)})^2 x_3^{(1)} + x_2^{(2)}x_5^{(2)}x_3^{(2)}x_3^{(1)} + x_2^{(2)}x_5^{(2)}x_4^{(1)}x_2^{(1)}}, \\
&c_{3_2} = \frac{cx_3^{(3)}(x_3^{(2)})^2 x_3^{(1)} + cx_2^{(2)}x_5^{(2)}x_3^{(2)}x_3^{(1)} + x_2^{(2)}x_5^{(2)}x_4^{(1)}x_2^{(1)}}{cx_3^{(3)}(x_3^{(2)})^2 x_3^{(1)} + x_2^{(2)}x_5^{(2)}x_3^{(2)}x_3^{(1)} + x_2^{(2)}x_5^{(2)}x_4^{(1)}x_2^{(1)}}, \\
&c_4 = \frac{cx_4^{(2)}x_4^{(1)}+x_3^{(3)}x_3^{(2)}}{x_4^{(2)}x_4^{(1)}+x_3^{(3)}x_3^{(2)}}, \\
&c_5 = \frac{cx_5^{(2)}x_5^{(1)}+x_3^{(2)}x_3^{(1)}}{x_5^{(2)}x_5^{(1)}+x_3^{(2)}x_3^{(1)}}. \\
\gamma_k(V_1(x)) &= 
\begin{cases}
\frac{(x_1^{(1)})^2}{x_2^{(2)}x_2^{(1)}}, &k=1,\\
\frac{(x_2^{(2)})^2 (x_2^{(1)})^2}{x_3^{(3)}x_3^{(2)}x_1^{(1)}x_3^{(1)}}, &k=2, \\ 
\frac{(x_3^{(3)})^2 (x_3^{(2)})^2 (x_3^{(1)})^2}{x_4^{(2)}x_2^{(2)}x_5^{(2)}x_4^{(1)}x_2^{(1)}x_5^{(1)}}, &k=3, \\ 
\frac{(x_4^{(2)})^2 (x_4^{(1)})^2}{x_3^{(3)}x_3^{(2)}x_3^{(1)}}, &k=4, \\ 
\frac{(x_5^{(2)})^2 (x_5^{(1)})^2}{x_3^{(3)}x_3^{(2)}x_3^{(1)}}, &k=5. \\ 
\end{cases}\\
\veps_k(V_1(x)) &= 
\begin{cases}
\frac{x_2^{(2)}}{x_1^{(1)}}, &k=1, \\ 
\frac{x_3^{(3)}}{x_2^{(2)}} \big( 1 + \frac{x_3^{(2)}x_1^{(1)}}{x_2^{(2)}x_2^{(1)}} \big), &k=2, \\ 
\frac{x_4^{(2)}}{x_3^{(3)}} \big( 1 + \frac{x_2^{(2)}x_5^{(2)}}{x_3^{(3)}x_3^{(2)}}+ \frac{x_2^{(2)}x_5^{(2)}x_4^{(1)}x_2^{(1)}}{x_3^{(3)}(x_3^{(2)})^2 x_3^{(1)}} \big), &k=3, \\ 
\frac{1}{x_4^{(2)}} \big( 1 + \frac{x_3^{(3)}x_3^{(2)}}{x_4^{(2)}x_4^{(1)}} \big), &k=4, \\ 
\frac{x_3^{(3)}}{x_5^{(2)}} \big( 1 + \frac{x_3^{(2)}x_3^{(1)}}{x_5^{(2)}x_5^{(1)}} \big), &k=5. \\ 
\end{cases} 
\end{align*}

By choosing ${\bf i} = (5,3,2,4,3,5,0,2,3,4)$, we also have the following explicit actions of \ $\bar{e_k}^c,  \ \bar{\gamma}_k, \ \bar{\veps}_k, \  k=0,2,3,4,5$ on $V_2(y)$  by Theorem \ref{schubert}.
\begin{align*}
\bar{e_k}^c(V_2(y)) &=
\begin{cases}
V_2(y_5^{(2)}, y_3^{(3)}, y_2^{(2)}, y_4^{(2)}, y_3^{(2)}, y_5^{(1)}, cy_0^{(1)}, y_2^{(1)},y_3^{(1)}, y_4^{(1)}), &k=0,\\
V_2(y_5^{(2)}, y_3^{(3)}, \bar{c}_2 y_2^{(2)}, y_4^{(2)}, y_3^{(2)}, y_5^{(1)}, y_0^{(1)}, \frac{c}{\bar{c}_2}y_2^{(1)},y_3^{(1)}, y_4^{(1)}), &k=2,\\
V_2(y_5^{(2)}, \bar{c}_{3_1} y_3^{(3)}, y_2^{(2)}, y_4^{(2)}, \bar{c}_{3_2}y_3^{(2)}, y_5^{(1)}, y_0^{(1)}, y_2^{(1)},\frac{c}{\bar{c}_{3_1} \bar{c}_{3_2}}y_3^{(1)}, y_4^{(1)}), &k=3,\\
V_2(y_5^{(2)}, y_3^{(3)}, y_2^{(2)}, \bar{c}_4y_4^{(2)}, y_3^{(2)}, y_5^{(1)}, y_0^{(1)}, y_2^{(1)},y_3^{(1)}, \frac{c}{\bar{c}_4}y_4^{(1)}), &k=4,\\
V_2(\bar{c}_5y_5^{(2)}, y_3^{(3)}, y_2^{(2)}, y_4^{(2)}, y_3^{(2)}, \frac{c}{\bar{c}_5}y_5^{(1)}, y_0^{(1)}, y_2^{(1)},y_3^{(1)}, y_4^{(1)}), &k=5,
\end{cases} \\
\text{where} & \\
&\bar{c}_2 =\frac{cy_2^{(2)}y_2^{(1)}+y_3^{(2)}y_0^{(1)}}{y_2^{(2)}y_2^{(1)}+y_3^{(2)}y_0^{(1)}}, \\
&\bar{c}_{3_1} = \frac{cy_3^{(3)}(y_3^{(2)})^2 y_3^{(1)} + y_2^{(2)}y_4^{(2)}y_3^{(2)}y_3^{(1)} + y_2^{(2)}y_4^{(2)}y_5^{(1)}x_2^{(1)}}{y_3^{(3)}(y_3^{(2)})^2 y_3^{(1)} + y_2^{(2)}y_4^{(2)}y_3^{(2)}y_3^{(1)} + y_2^{(2)}y_4^{(2)}y_5^{(1)}x_2^{(1)}}, \\
&\bar{c}_{3_2} = \frac{cy_3^{(3)}(y_3^{(2)})^2 y_3^{(1)} + cy_2^{(2)}y_4^{(2)}y_3^{(2)}y_3^{(1)} + y_2^{(2)}y_4^{(2)}y_5^{(1)}x_2^{(1)}}{cy_3^{(3)}(y_3^{(2)})^2 y_3^{(1)} + y_2^{(2)}y_4^{(2)}y_3^{(2)}y_3^{(1)} + y_2^{(2)}y_4^{(2)}y_5^{(1)}x_2^{(1)}}, \\
&\bar{c}_4 = \frac{cy_5^{(2)}y_5^{(1)}+y_3^{(3)}y_3^{(2)}}{y_5^{(2)}y_5^{(1)}+y_3^{(3)}y_3^{(2)}}, \\
&\bar{c}_5 = \frac{cy_4^{(2)}y_4^{(1)}+y_3^{(2)}y_3^{(1)}}{y_4^{(2)}y_4^{(1)}+y_3^{(2)}y_3^{(1)}}. \\
\end{align*}
\begin{align*}
\bar{\gamma}_k(V_2(y)) &= 
\begin{cases}
\frac{(y_0^{(1)})^2}{y_2^{(2)}y_2^{(1)}}, &k=0,\\
\frac{(y_2^{(2)})^2 (y_2^{(1)})^2}{y_3^{(3)}y_3^{(2)}y_0^{(1)}y_3^{(1)}}, &k=2, \\ 
\frac{(y_3^{(3)})^2 (y_3^{(2)})^2 (y_3^{(1)})^2}{y_5^{(2)}y_2^{(2)}y_4^{(2)}y_5^{(1)}y_2^{(1)}y_4^{(1)}}, &k=3, \\ 
\frac{(y_4^{(2)})^2 (y_4^{(1)})^2}{y_3^{(3)}y_3^{(2)}y_3^{(1)}}, &k=4, \\ 
\frac{(y_5^{(2)})^2 (y_5^{(1)})^2}{y_3^{(3)}y_3^{(2)}y_3^{(1)}}, &k=5. \\ 
\end{cases}\\
\bar{\veps}_k(V_2(y)) &= 
\begin{cases}
\frac{y_2^{(2)}}{y_0^{(1)}}, &k=0, \\ 
\frac{y_3^{(3)}}{y_2^{(2)}} \big( 1 + \frac{y_3^{(2)}y_0^{(1)}}{y_2^{(2)}y_2^{(1)}} \big), &k=2, \\ 
\frac{y_5^{(2)}}{y_3^{(3)}} \big( 1 + \frac{y_2^{(2)}y_4^{(2)}}{y_3^{(3)}y_3^{(2)}}+ \frac{y_2^{(2)}y_4^{(2)}y_5^{(1)}y_2^{(1)}}{y_3^{(3)}(y_3^{(2)})^2 y_3^{(1)}} \big), &k=3, \\ 
\frac{y_3^{(3)}}{y_4^{(2)}} \big( 1 +\frac{y_3^{(2)}y_3^{(1)}}{y_4^{(2)}y_4^{(1)}} \big), &k=4, \\ 
\frac{1}{y_5^{(2)}} \big( 1 + \frac{y_3^{(3)}y_3^{(2)}}{y_5^{(2)}y_5^{(1)}} \big), &k=5. \\ 
\end{cases} 
\end{align*}
\begin{prop}\label{relation2-3} The following relations hold.\\
(i) $\bar{\sigma} e_2^c = \bar{e_3}^c \bar{\sigma}$ ,\\
(ii) $\bar{\sigma} e_3^c = \bar{e_2}^c \bar{\sigma}$.
\end{prop}
\begin{proof} We will prove the relation $(i)$ and relation $(ii)$ can be shown similarly.  Set $e_2^c (V_1(x)) = V_1(z)$, $\bar{\sigma}(V_1(z)) = V_2(y')$,  $\bar{\sigma} (V_1(x)) = V_2(y)$ and $\bar{e_3}^c (V_2(y)) = V_2(w)$. We need to show that $y_m^{(l)'} = w_m^{(l)}$ for $(l,m) \in \{(1,0), (1,2), (1,3), (1,4), (1,5), (2,2),\\ (2,3), (2,4), (2,5), (3,3)\}$. Let us check this equality for $l=1$ and the rest can be verified similarly. 
\begin{itemize}
\item $y_0^{(1)'} = \frac{z_4^{(2)}z_4^{(1)}}{z_5^{(2)}z_5^{(1)}} = \frac{x_4^{(2)}x_4^{(1)}}{x_5^{(2)}x_5^{(1)}} = y_0^{(1)} = w_0^{(1)}.$
\item $y_2^{(1)'} = \frac{z_3^{(3)}z_3^{(2)}z_3^{(1)}}{(z_5^{(2)})^2 z_5^{(1)}} \big(\frac{z_3^{(3)}z_3^{(2)}}{z_4^{(2)}z_5^{(2)}} + \frac{z_4^{(1)}}{z_5^{(2)}} \big)^{-1} = \frac{x_3^{(3)}x_3^{(2)}x_3^{(1)}}{(x_5^{(2)})^2 x_5^{(1)}} \big(\frac{x_3^{(3)}x_3^{(2)}}{x_4^{(2)}x_5^{(2)}} + \frac{x_4^{(1)}}{x_5^{(2)}} \big)^{-1} \\= y_2^{(1)} = w_2^{(1)}$.
\item $y_3^{(1)'} = \frac{z_2^{(2)}z_2^{(1)}}{z_5^{(2)}} \big(\frac{z_2^{(2)}z_3^{(1)}}{z_3^{(3)}} + \frac{z_3^{(2)}z_3^{(1)}}{z_5^{(2)}} + \frac{z_2^{(2)}z_2^{(1)}}{z_4^{(2)}} + \frac{z_2^{(2)}z_4^{(1)}z_2^{(1)}}{z_3^{(3)}z_3^{(2)}} \big)^{-1} \\
= \frac{cx_2^{(2)}x_2^{(1)}}{x_5^{(2)}} \Big(\frac{x_2^{(2)}x_3^{(1)}(cx_2^{(2)}x_2^{(1)}+x_3^{(3)}x_1^{(1)})}{x_3^{(3)}(x_2^{(2)}x_2^{(1)}+x_3^{(3)}x_1^{(1)})} + \frac{x_3^{(2)}x_3^{(1)}}{x_5^{(2)}} + \frac{cx_2^{(2)}x_2^{(1)}}{x_4^{(2)}} + \frac{cx_2^{(2)}x_4^{(1)}x_2^{(1)}}{x_3^{(3)}x_3^{(2)}} \Big)^{-1} \\
= y_3^{(1)} \cdot \frac{c(y_3^{(3)}(y_3^{(2)})^2y_3^{(1)} +y_2^{(2)}y_4^{(2)}y_3^{(2)}y_3^{(1)} +y_2^{(2)}y_4^{(2)}y_5^{(1)}y_2^{(1)})}{cy_3^{(3)}(y_3^{(2)})^2y_3^{(1)} +cy_2^{(2)}y_4^{(2)}y_3^{(2)}y_3^{(1)} +y_2^{(2)}y_4^{(2)}y_5^{(1)}y_2^{(1)} }  = w_3^{(1)}$.
\item $y_4^{(1)'} = \big(\frac{z_5^{(2)}}{z_3^{(3)}} + \frac{z_3^{(2)}}{z_2^{(2)}} + \frac{z_2^{(1)}}{z_1^{(1)}} \big)^{-1}= \big(\frac{x_5^{(2)}}{x_3^{(3)}} + \frac{x_3^{(2)}}{x_2^{(2)}} + \frac{x_2^{(1)}}{x_1^{(1)}} \big)^{-1}= y_4^{(1)} = w_4^{(1)}$.
\item $y_5^{(1)'} = \frac{1}{z_5^{(2)}} = \frac{1}{x_5^{(2)}} = y_5^{(1)} = w_5^{(1)}$.
\end{itemize}
\end{proof}
In order to give $\mathcal{V}_1$ a $\ge=D_5^{(1)}$-geometric crystal structure, we need to define the actions of $e_0^c, \gamma_0$, and $\veps_0$ on $V_1(x)$. We use the $\ge_1$-geometric crystal structure on $\cV_2$ to define the action of $e_0^c$, $\gamma_0$, and $\veps_0$ on $V_1(x)$ as follows. 
\begin{align}
e_0^c (V_1(x)) &:= \bar{\sigma}^{-1} \circ \overline{e_{\sigma{(0)}}}^c \circ \bar{\sigma} (V_1(x))= \bar{\sigma}^{-1} \circ \bar{e}_5^c(V_2(y)), \\
\gamma_0 (V_1(x)) &:= \bar{\sigma}^{-1} \circ\overline{ \gamma_{\sigma{(0)}}} \circ \bar{\sigma} (V_1(x))= \bar{\sigma}^{-1} \circ\overline{ \gamma_5} (V_2(y)). \label{gamma_0}\\
\veps_0 (V_1(x)) &:= \bar{\sigma}^{-1} \circ\overline{ \veps_{\sigma{(0)}}} \circ \bar{\sigma} (V_1(x))= \bar{\sigma}^{-1} \circ\overline{ \veps_5}(V_2(y)), \label{veps_0}
\end{align}
Set $B=B_x = \frac{x_2^{(2)}x_3^{(1)}}{x_3^{(3)}} + \frac{x_3^{(2)}x_3^{(1)}}{x_5^{(2)}}, C=C_x =  \frac{x_2^{(2)}x_2^{(1)}}{x_4^{(2)}} + \frac{x_2^{(2)}x_4^{(1)}x_2^{(1)}}{x_3^{(3)}x_3^{(2)}}$ and $A=A_x = B_x+C_x$. The following is one of the main results in this paper.

\begin{theorem} The algebraic variety $\cV = \cV_1 = \{V_1(x), e_k^c, \gamma_k, \veps_k \mid k \in I\}$ is a positive geometric crystal for the affine Lie algebra $\ge = D_5^{(1)}$ with the $e_0^c$, $\gamma_0$, and $\veps_0$ actions on $V_1(x)$ given by:
\begin{align*}
\gamma_0(V_1(x)) &= \frac{1}{x_2^{(2)}x_2^{(1)}},   \hspace{1.0cm}  \veps_0(V_1(x)) = x_5^{(1)} + A,  \\
e_0^{c}(V_1(x)) &= V_1(x') =  V_1(x_4^{(2)'}, x_3^{(3)'}, \ldots, x_5^{(1)'}) \ \text{where}
\end{align*}
\begin{align*} 
x_4^{(2)'} &= x_4^{(2)} \cdot \frac{c(x_5^{(1)}+B+\frac{x_2^{(2)}x_4^{(1)}x_2^{(1)}}{x_3^{(3)}x_3^{(2)}})+\frac{x_2^{(2)}x_2^{(1)}}{x_4^{(2)}}}{c(x_5^{(1)}+A)},  \hspace{-.2cm} \\ 
x_3^{(3)'} &= x_3^{(3)} \cdot \frac{c(x_5^{(1)}+\frac{x_3^{(2)}x_3^{(1)}}{x_5^{(2)}})+\frac{x_2^{(2)}x_3^{(1)}}{x_3^{(3)}}+C}{c(x_5^{(1)}+A)}, \hspace{-.2cm} \\ 
x_2^{(2)'} &= \frac{x_2^{(2)}}{c}, \hspace{0.5cm}  x_5^{(2)'} = x_5^{(2)} \cdot \frac{cx_5^{(1)}+A}{c(x_5^{(1)}+A)}, \\
x_3^{(2)'} &= x_3^{(2)} \cdot \frac{c(x_5^{(1)}+B)+C}{c(c(x_5^{(1)}+\frac{x_3^{(2)}x_3^{(1)}}{x_5^{(2)}})+\frac{x_2^{(2)}x_3^{(1)}}{x_3^{(3)}}+C)},  \hspace{-.2cm} \\
x_4^{(1)'} &= x_4^{(1)} \cdot \frac{x_5^{(1)}+A}{c(x_5^{(1)}+B+\frac{x_2^{(2)}x_4^{(1)}x_2^{(1)}}{x_3^{(3)}x_3^{(2)}})+\frac{x_2^{(2)}x_2^{(1)}}{x_4^{(2)}}}, \hspace{0.5cm} x_1^{(1)'} = \frac{x_1^{(1)}}{c},\\ 
x_2^{(1)'} &= \frac{x_2^{(1)}}{c}, \hspace{0.5cm} x_3^{(1)'} = x_3^{(1)} \cdot \frac{x_5^{(1)}+A}{c(x_5^{(1)}+B)+C}, \\
x_5^{(1)'} &= x_5^{(1)} \cdot \frac{x_5^{(1)}+A}{cx_5^{(1)}+A}.
\end{align*}
\end{theorem}

\begin{proof} Since $\cV= \cV_1$ is a positive geometric crystal for $\ge_0$, to show that it is a positive geometric crystal for 
$\ge = D_5^{(1)}$ by Definition \ref{geometric}, it suffices to show that the following relations involving the $0$-action hold:
\begin{enumerate}
	\item $\gamma_0 (e_k^c (V_1(x))) = c^{a_{k0}} \gamma_0 (V_1(x))$ for all $k \in I$ \\
	\item $\gamma_k (e_0^c (V_1(x))) = c^{a_{0k}} \gamma_k (V_1(x))$ for all $k \in I$ \\
	\item $\veps_0 (e_0^c (V_1(x))) = c^{-1} \veps_0 (V_1(x))$\\
	\item $e_0^{c_1} e_k^{c_2} = e_k^{c_2} e_0^{c_1}$ for all $k \in \{1,3,4,5\}$ \\
	\item $e_0^{c_1} e_2^{c_1 c_2} e_0^{c_2}= e_2^{c_2} e_0^{c_1 c_2} e_2^{c_1}$ \\
\end{enumerate}
Note that \ $x_2^{(2)'}x_2^{(1)'}= c^{-2}x_2^{(2)}x_2^{(1)}, \  x_3^{(3)'}x_3^{(2)'}x_3^{(1)'}= c^{-2}x_3^{(3)}x_3^{(2)}x_3^{(1)}, \  x_4^{(2)'}x_4^{(1)'}= c^{-1}x_4^{(2)}x_4^{(1)}, \ x_5^{(2)'}x_5^{(1)'}= c^{-1}x_5^{(2)}x_5^{(1)}$.
Relations (1) and (2) follows easily from the defined actions. For example,
\begin{align*}
\gamma_0 (e_0^c (V_1(x))) &= \frac{1}{x_2^{(2)'}x_2^{(1)'}} =  \frac{1}{c^{-2}x_2^{(2)}x_2^{(1)}} = c^2 \gamma_0 (V_1(x)))\\
\gamma_0 (e_2^c (V_1(x))) &= \frac{1}{c_2x_2^{(2)}cc_2^{-1}x_2^{(1)}} =  c^{-1} \gamma_0 (V_1(x))) \\
\gamma_1 (e_0^c (V_1(x))) &= \frac{(x_1^{(1)'})^2}{x_2^{(2)'}x_2^{(1)'}} = \frac{(c^{-1}x_1^{(1)})^2}{c^{-2}x_2^{(2)}x_2^{(1)}} =  c^0 \gamma_1 (V_1(x)))\\
\end{align*}
Now we consider relation (3). We have
\begin{align*}
\veps_0 (e_0^c (V_1(x))) 
	&= x_5^{(1)'} + \frac{x_2^{(2)'}x_3^{(1)'}}{x_3^{(3)'}} + \frac{x_3^{(2)'}x_3^{(1)'}}{x_5^{(2)'}} +\frac{x_2^{(2)'}x_2^{(1)'}}{x_4^{(2)'}} + \frac{x_2^{(2)'}x_4^{(1)'}x_2^{(1)'}}{x_3^{(3)'}x_3^{(2)'}}\\
	&=(x_5^{(1)} + A) \Big[\frac{x_5^{(1)}}{cx_5^{(1)}+A} + \\
   	& \frac{x_3^{(1)}}{x_3^{(3)}x_5^{(2)}} \cdot \frac{x_5^{(1)}+A}{c(x_5^{(1)}+\frac{x_3^{(2)}x_3^{(1)}}{x_5^{(2)}})+ \frac{x_2^{(2)}x_3^{(1)}}{x_3^{(3)}} +C} \cdot \big(\frac{x_5^{(2)}x_2^{(2)}}{c(x_5^{(1)}+B)+C }  + \frac{x_3^{(3)}x_3^{(2)}}{cx_5^{(1)}+A}\big) + \\
    	& \frac{x_2^{(2)}x_2^{(1)}}{cx_4^{(2)}x_3^{(3)}x_3^{(2)}} \cdot \frac{x_3^{(3)}x_3^{(2)}(c(x_5^{(1)}+B)+C)+cx_4^{(2)}x_4^{(1)}(x_5^{(1)}+A)}{\big(c(x_5^{(1)}+B)+C \big) \big(c(x_5^{(1)}+B+\frac{x_2^{(2)}x_4^{(1)}x_2^{(1)}}{x_3^{(3)}x_3^{(2)}})+ \frac{x_2^{(2)}x_2^{(1)}}{x_4^{(2)}} \big)}  \Big] \\
	&=(x_5^{(1)} + A) \Big[\frac{x_5^{(1)}}{cx_5^{(1)}+A} + \\
   	& \frac{x_3^{(1)}}{x_3^{(3)}x_5^{(2)}} \cdot \frac{(x_5^{(1)}+A)(x_3^{(3)}x_3^{(2)}+x_5^{(2)}x_2^{(2)})}{\big(c(x_5^{(1)}+B)+C\big)\big(cx_5^{(1)}+A\big)} + \frac{x_2^{(2)}x_2^{(1)}}{cx_4^{(2)}x_3^{(3)}x_3^{(2)}} \cdot \frac{x_3^{(3)}x_3^{(2)}+x_4^{(2)}x_4^{(1)}}{c(x_5^{(1)}+B)+C}  \Big] \\	
	&=c^{-1} (x_5^{(1)} + A)  \Big[\frac{cx_5^{(1)}(c(x_5^{(1)}+B)+C)+c(x_5^{(1)}+A)B +(cx_5^{(1)}+A)C}{(cx_5^{(1)}+A)(c(x_5^{(1)}+B)+C) } \Big] \\
&=c^{-1} (x_5^{(1)} + A) = c^{-1} \veps_0 (V_1(x)),
\end{align*} 
since $c(x_5^{(1)}+A)B + (cx_5^{(1)}+A)C = cx_5^{(1)}A + cAB +AC = A(c(x_5^{(1)}+B) + C)$. Next we consider relation (4). Since $\mathcal{V}_2$ is a $\mathfrak{g}_1$-geometric crystal, we have $\bar{e_5}^{c_1} \bar{e_2}^{c_2}=\bar{e_2}^{c_2} \bar{e_5}^{c_1}$. Then by Proposition \ref{relation2-3}, we have 
\begin{equation*}
e_0^{c_1} e_3^{c_2} = \bar{\sigma}^{-1} \bar{e_{5}}^{c_1}  \bar{\sigma} e_3^{c_2} 
= \bar{\sigma}^{-1} \bar{e_{5}}^{c_1}  \bar{e_2}^{c_2} \bar{\sigma}
= \bar{\sigma}^{-1}\bar{e_2}^{c_2} \bar{e_5}^{c_1} \bar{\sigma}
= e_3^{c_2} \bar{\sigma}^{-1}  \bar{e_5}^{c_1} \bar{\sigma}
=e_3^{c_2} e_0^{c_1}.
\end{equation*}
Next we show relation (4) for $k=1$. It can be shown for $k = 4,5$ similarly. We need to show that
$$e_0^{c_1} e_1^{c_2} (V_1(x)) = e_1^{c_2} e_0^{c_1} (V_1(x)).$$
Set $e_1^{c_2} (V_1(x)) = V_1(z)$, $e_0^{c_1} (V_1(z)) = V_1(z')$, $e_0^{c_1} (V_1(x)) =V_1(x')$ and  $e_1^{c_2}  (V_1(x')) =V_1(u)$. We have to show that 
$$z_m^{(l)'} = u_m^{(l)}$$
for $(l,m) \in\{(1,1), (1,2), (1,3), (1,4), (1,5), (2,2), (2,3), (2,4), (2,5), (3,3)\}$. Observe that
$B_z = \frac{z_2^{(2)}z_3^{(1)}}{z_3^{(3)}} + \frac{z_3^{(2)}z_3^{(1)}}{z_5^{(2)}} = B_x = B$, $C_z = \frac{z_2^{(2)}z_2^{(1)}}{z_4^{(2)}} + \frac{z_2^{(2)}z_4^{(1)}z_2^{(1)}}{z_3^{(3)}z_3^{(2)}} = C_x = C$. Hence $A_z = A$. Now we have the following.
\begin{itemize}
\item $z_1^{(1)'} = \frac{z_1^{(1)}}{c_1} = \frac{c_2 x_1^{(1)}}{c_1} = c_2 x_1^{(1)'} = u_1^{(1)}$.
\item $z_2^{(1)'} = \frac{z_2^{(1)}}{c_1} = \frac{x_2^{(1)}}{c_1} = x_2^{(1)'} =  u_2^{(1)}$.
\item $z_3^{(1)'} = z_3^{(1)} \cdot \frac{z_5^{(1)}+A_z}{c_1(z_5^{(1)}+B_z)+C_z} = x_3^{(1)} \cdot \frac{x_5^{(1)}+A}{c_1(x_5^{(1)}+B)+C} = x_3^{(1)'} = u_3^{(1)}$.
\item $z_4^{(1)'} = z_4^{(1)} \cdot \frac{z_5^{(1)}+A_z}{c_1(z_5^{(1)}+B_z+\frac{z_2^{(2)}z_4^{(1)}z_2^{(1)}}{z_3^{(3)}z_3^{(2)}})+\frac{z_2^{(2)}z_2^{(1)}}{z_4^{(2)}}} = x_4^{(1)} \cdot \frac{x_5^{(1)}+A}{c_1(x_5^{(1)}+B+\frac{x_2^{(2)}x_4^{(1)}x_2^{(1)}}{x_3^{(3)}x_3^{(2)}})+\frac{x_2^{(2)}x_2^{(1)}}{x_4^{(2)}}}\\= x_4^{(1)'} = u_4^{(1)}$.
\item $z_5^{(1)'} = z_5^{(1)}\cdot \frac{z_5^{(1)}+A_z}{c_1z_5^{(1)}+A_z} = x_5^{(1)} \cdot \frac{x_5^{(1)}+A}{c_1x_5^{(1)}+A}= x_5^{(1)'} = u_5^{(1)}$.
\item $z_2^{(2)'} = \frac{z_2^{(2)}}{c_1} = \frac{x_2^{(2)}}{c_1} = x_2^{(2)'} =  u_2^{(2)}$.
\item $z_3^{(2)'} = z_3^{(2)} \cdot \frac{c_1(z_5^{(1)}+B_z)+C_z}{c_1(c_1(z_5^{(1)}+\frac{z_3^{(2)}z_3^{(1)}}{z_5^{(2)}})+\frac{z_2^{(2)}z_3^{(1)}}{z_3^{(3)}}+C_z)} = x_3^{(2)} \cdot \frac{c_1(x_5^{(1)}+B)+C}{c_1(c_1(x_5^{(1)}+\frac{x_3^{(2)}x_3^{(1)}}{x_5^{(2)}})+\frac{x_2^{(2)}x_3^{(1)}}{x_3^{(3)}}+C)} \\ = x_3^{(2)'} = u_3^{(2)}$.
\item $z_4^{(2)'} = z_4^{(2)} \cdot \frac{c_1(z_5^{(1)}+B_z+\frac{z_2^{(2)}z_4^{(1)}z_2^{(1)}}{z_3^{(3)}z_3^{(2)}})+\frac{z_2^{(2)}z_2^{(1)}}{z_4^{(2)}}}{c_1(z_5^{(1)}+A_z)} = x_4^{(2)} \cdot \frac{c_1(x_5^{(1)}+B+\frac{x_2^{(2)}x_4^{(1)}x_2^{(1)}}{x_3^{(3)}x_3^{(2)}})+\frac{x_2^{(2)}x_2^{(1)}}{x_4^{(2)}}}{c_1(x_5^{(1)}+A)} \\ = x_4^{(2)'} = u_4^{(2)}$.
\item $z_5^{(2)'} = z_5^{(2)} \cdot \frac{c_1z_5^{(1)}+A_z}{c_1(z_5^{(1)}+A_z)} = x_5^{(2)} \cdot \frac{c_1x_5^{(1)}+A}{c_1(x_5^{(1)}+A)} =  x_5^{(2)'} =  u_5^{(2)}$.
\item $z_3^{(3)'} = z_3^{(3)} \cdot \frac{c_1(z_5^{(1)}+\frac{z_3^{(2)}z_3^{(1)}}{z_5^{(2)}})+\frac{z_2^{(2)}z_3^{(1)}}{z_3^{(3)}}+C_z}{c_1(z_5^{(1)}+A_z)} = x_3^{(3)} \cdot \frac{c_1(x_5^{(1)}+\frac{x_3^{(2)}x_3^{(1)}}{x_5^{(2)}})+\frac{x_2^{(2)}x_3^{(1)}}{x_3^{(3)}}+C}{c_1(x_5^{(1)}+A)} \\ = x_3^{(3)'}= u_3^{(3)}$.
\end{itemize}

Finally to show relation (5) we observe that $\bar{e_5}^{c_1} \bar{e_3}^{c_1c_2} \bar{e_5}^{c_2}=\bar{e_3}^{c_2} \bar{e_5}^{c_1c_2} \bar{e_3}^{c_1}$ since $\mathcal{V}_2$ is a $\mathfrak{g}_1$-geometric crystal. Hence by Proposition \ref{relation2-3}, we have 
\begin{align*}
e_0^{c_1} e_2^{c_1c_2}e_0^{c_2} &= \bar{\sigma}^{-1} \bar{e_{5}}^{c_1}  \bar{\sigma} e_2^{c_1c_2} \bar{\sigma}^{-1} \bar{e_{5}}^{c_2}  \bar{\sigma} \\
&= \bar{\sigma}^{-1} \bar{e_{5}}^{c_1}  \bar{e_3}^{c_1c_2} \bar{e_{5}}^{c_2}  \bar{\sigma} \,
= \bar{\sigma}^{-1} \bar{e_3}^{c_2} \bar{e_5}^{c_1c_2} \bar{e_3}^{c_1}  \bar{\sigma} \\
&= e_2^{c_2} \bar{\sigma}^{-1}  \bar{e_5}^{c_1c_2} \bar{\sigma} e_2^{c_1}  \,
=e_2^{c_2} e_0^{c_1c_2}e_2^{c_1},
\end{align*}
which completes the proof.
\end{proof}

\section{Perfect Crystals of type \bf{$D_5^{(1)}$}}
For a positive integer $l$, we consider the sets $B^{5,l}$ and $B^{5,\infty}$ as follows. 
\begin{align*}
B^{5,l} &= 
     \left \{ b= (b_{ij})_{\scriptsize{\begin{array}{l}i \leq j \leq i+4,\\ 1 \leq i \leq 5\end{array}}} \middle|
\begin{aligned}
&b_{ij} \in \mZ_{\geq 0},\ \sum_{j=i}^{i+4} b_{ij} = l,\ 1 \leq i \leq 5,\\ 
&\sum_{j=i}^{5-t} b_{ij} = \sum_{j=i+t}^{4+t} b_{i+t,j},\ 1 \leq i, t \leq 4,\\
& \sum_{j=i}^{t} b_{ij} \geq \sum_{j=i+1}^{t+1} b_{i+1,j},\ 1 \leq i \leq t \leq 4
\end{aligned} \right\}, \\
B^{5,\infty} &= 
    \left \{ b= (b_{ij})_{\scriptsize{\begin{array}{l}i \leq j \leq i+4,\\ 1 \leq i \leq 5\end{array}}} \middle|
\begin{aligned}
&b_{ij} \in \mZ,\ \sum_{j=i}^{i+4} b_{ij} = 0,\ 1 \leq i \leq 5, \\
&\sum_{j=i}^{5-t} b_{ij} = \sum_{j=i+t}^{4+t} b_{i+t,j},\ 1 \leq i, t \leq 4
 \end{aligned} \right\}.
\end{align*}
For $\cB = B^{5,l} \, \text{or} \, B^{5,\infty}$ we define the maps $\e_k, \f_k : \cB \longrightarrow \cB \cup \{0\}$, $\veps_k , \vphi_k : \cB \longrightarrow \mZ$, $0 \leq k \leq 5$ and $\text{wt} : \cB \longrightarrow P_{cl}$, as follows. First we define conditions $(E_j), \ 1 \leq j \leq 5$:
\begin{align*}
(E_1) 	& \hspace{5pt}b_{22} > b_{13}+b_{14}+b_{44}, b_{22}+b_{23} > b_{14}+b_{33}+b_{34}, \\
		& \hspace{15pt}b_{22} > b_{14}+b_{33}+(b_{13}-b_{24})_{+}, \\
(E_2) 	&\hspace{5pt}b_{33}+b_{34}> b_{13}+b_{23}+b_{44}, b_{14}+b_{33}+b_{34} \geq b_{22}+b_{23},\\
		& \hspace{15pt}b_{34} > b_{23}+(b_{13}-b_{24})_{+}, \\
(E_3) 	&\hspace{5pt}b_{33} > b_{13}+b_{44}, b_{24} > b_{13},\\
		& \hspace{15pt}b_{14}+b_{33} \geq b_{22}+(b_{14}- b_{22}-b_{23}+b_{33}+b_{34})_{+}, \\
(E_4) 	&\hspace{5pt}b_{33} > b_{24}+b_{44}, b_{13} \geq b_{24},\\
		& \hspace{15pt}b_{13}+b_{14}+b_{33} \geq b_{22}+b_{24}+(b_{14}-b_{22}- b_{23}+b_{33}+b_{34})_{+}, \\
(E_5) 	&\hspace{5pt}b_{13}+b_{44} \geq b_{33}+(b_{13}-b_{24})_{+}, \\
		& \hspace{15pt} b_{13}+b_{14}+b_{44} \geq b_{22}+(b_{14}-b_{22}- b_{23}+b_{33}+b_{34})_{+},
\end{align*}
where $(x)_{+}=\text{max}(x,0)$. Then we define conditions $(F_j) \ (1 \leq j \leq 5)$ by replacing $>$ (resp. $\geq$) with $\geq$ (resp. $>$) in $(E_j)$.  Let $b=(b_{ij}) \in \cB$. Then for $\tilde{e_k}(b) = (b'_{ij})$ where 

\begin{align*}
k=0 &: 
	\begin{cases} 
	b'_{11} = b_{11} - 1,b'_{15} = b_{15} + 1, b'_{22} = b_{22} - 1, b'_{26} = b_{26} + 1, \\ 
	b'_{35} = b_{35} - 1,  b'_{37} = b_{37} + 1, b'_{46} = b_{46} - 1, b'_{48} = b_{48} + 1, \\ 
	b'_{57} = b_{57} - 1, b'_{59} = b_{59} + 1 \ \text{if} \ (E1)  \vspace{1pt}\\ 
	b'_{11} = b_{11} - 1,b'_{14} = b_{14} + 1, b'_{22} = b_{22} - 1, b'_{25} = b_{25} + 1, \\ 
	b'_{34} = b_{34} - 1,  b'_{37} = b_{37} + 1, b'_{45} = b_{45} - 1, b'_{48} = b_{48} + 1, \\ 
	b'_{57} = b_{57} - 1, b'_{59} = b_{59} + 1 \ \text{if} \ (E2) \vspace{1pt}\\ 
	b'_{11} = b_{11} - 1,b'_{14} = b_{14} + 1, b'_{22} = b_{22} - 1,b'_{23} = b_{23} + 1, \\
	b'_{24} = b_{24} - 1,  b'_{25} = b_{25} + 1, b'_{33} = b_{33} - 1, b'_{37} = b_{37} + 1, \\ 
	b'_{45} = b_{45} - 1, b'_{46} = b_{46} + 1, b'_{47} = b_{47} - 1, b'_{48} = b_{48} + 1,\\ 
	b'_{56} = b_{56} - 1, b'_{59} = b_{59} + 1 \ \text{if} \ (E3) \vspace{1pt}\\ 
	b'_{11} = b_{11} - 1,b'_{13} = b_{13} + 1, b'_{22} = b_{22} - 1, b'_{25} = b_{25} + 1, \\ 
	b'_{33} = b_{33} - 1,  b'_{36} = b_{36} + 1, b'_{45} = b_{45} - 1, b'_{48} = b_{48} +1, \\ 
	b'_{56} = b_{56} -1, b'_{59} = b_{59} + 1 \ \text{if} \ (E4) \vspace{1pt}\\ 
	b'_{11} = b_{11} - 1,b'_{13} = b_{13} + 1, b'_{22} = b_{22} - 1, b'_{24} = b_{24} + 1,\\  
	b'_{33} = b_{33} - 1,  b'_{35} = b_{35} + 1, b'_{44} = b_{44} - 1, b'_{48} = b_{48} + 1, \\ 
	b'_{55} = b_{55} - 1, b'_{59} = b_{59} + 1 \ \text{if} \ (E5) 
	\end{cases}
	\\
k=1 &: b'_{11} =b_{11} +1, b'_{12} =b_{12} -1, b'_{58} =b_{58} +1, b'_{59} =b_{59} -1 \\
k=2 &: 
	\begin{cases} 
	b'_{12} =b_{12} +1, b'_{13} =b_{13} -1, b'_{47} =b_{47} +1, b'_{48} =b_{48} -1 \ \text{if} \ b_{12} \geq b_{23}\\
	b'_{22} =b_{22} +1, b'_{23} =b_{23} -1, b'_{57} =b_{57} +1, b'_{58} =b_{58} -1 \ \text{if} \ b_{12} < b_{23} 
	\end{cases}
	\\
k=3 &: 
	\begin{cases} 
	b'_{13} =b_{13} +1, b'_{14} =b_{14} -1, b'_{36} =b_{36} +1, b'_{37} =b_{37} -1  \\
	 \hspace{1cm} \text{if} \ b_{13} \geq b_{24}, b_{13}+b_{23} \geq b_{24}+b_{34} \\ 
	 b'_{23} =b_{23} +1, b'_{24} =b_{24} -1, b'_{46} =b_{46} +1, b'_{47} =b_{47} -1  \\ 
	 \hspace{1cm} \text{if} \ b_{13} <  b_{24}, b_{23} \geq b_{34}\\ 
	 b'_{33} =b_{33} +1, b'_{34} =b_{34} -1, b'_{56} =b_{56} +1, b'_{57} =b_{57} -1  \\
	 \hspace{1cm} \text{if} \ b_{23} <  b_{34}, b_{13}+b_{23} < b_{24}+b_{34} 
	 \end{cases}
	 \\
k=4 &: 
	\begin{cases} 
	b'_{14} =b_{14} +1, b'_{15} =b_{15} -1, b'_{25} =b_{25} +1, b'_{26} =b_{26} -1 \\ 
	\hspace{1cm} \text{if} \ b_{14}+b_{33}+b_{34} \geq b_{22}+b_{23}\\ 
	b'_{34} =b_{34} +1, b'_{35} =b_{35} -1, b'_{45} =b_{45} +1, b'_{46} =b_{46} -1 \\ 
	\hspace{1cm} \text{if} \ b_{14}+b_{33}+b_{34} < b_{22}+b_{23} 
	\end{cases}
	\\
k=5 &: 
	\begin{cases} 
	b'_{24} =b_{24} +1, b'_{25} =b_{25} -1, b'_{35} =b_{35} +1, b'_{36} =b_{36} -1 \\ 
	\hspace{1cm}  \text{if} \ b_{24}+b_{44} \geq b_{33}\\ 
	b'_{44} =b_{44} +1, b'_{45} =b_{45} -1, b'_{55} =b_{55} +1, b'_{56} =b_{56} -1 \\ 
	\hspace{1cm}  \text{if} \ b_{24}+b_{44} < b_{33} 
	\end{cases} 
\end{align*}  and $b'_{ij} = b_{ij}$ otherwise.\\

Also $\tilde{f_k}(b) = (b'_{ij})$ where
\begin{align*}
k=0 &: 
	\begin{cases} 
	b'_{11} = b_{11} + 1,b'_{15} = b_{15} - 1, b'_{22} = b_{22} + 1, b'_{26} = b_{26} - 1, \\ 
	b'_{35} = b_{35} + 1,  b'_{37} = b_{37} - 1, b'_{46} = b_{46} + 1, b'_{48} = b_{48} - 1, \\ 
	b'_{57} = b_{57} + 1, b'_{59} = b_{59} - 1 \ \text{if} \ (F1)  \vspace{1pt}\\ 
	b'_{11} = b_{11} + 1,b'_{14} = b_{14} - 1, b'_{22} = b_{22} + 1, b'_{25} = b_{25} - 1, \\ 
	b'_{34} = b_{34} + 1,  b'_{37} = b_{37} - 1, b'_{45} = b_{45} + 1, b'_{48} = b_{48} - 1, \\ 
	b'_{57} = b_{57} + 1, b'_{59} = b_{59} - 1 \ \text{if} \ (F2) \vspace{1pt}\\ 
	b'_{11} = b_{11} + 1,b'_{14} = b_{14} - 1, b'_{22} = b_{22} + 1,b'_{23} = b_{23} - 1,\\ 
	b'_{24} = b_{24} + 1,  b'_{25} = b_{25} - 1, b'_{33} = b_{33} + 1, b'_{37} = b_{37} - 1, \\ 
	b'_{45} = b_{45} + 1, b'_{46} = b_{46} - 1, b'_{47} = b_{47} + 1, b'_{48} = b_{48} - 1, \\ 
	b'_{56} = b_{56} + 1, b'_{59} = b_{59} - 1 \ \text{if} \ (F3) \vspace{1pt}\\ 
	b'_{11} = b_{11} + 1,b'_{13} = b_{13} - 1, b'_{22} = b_{22} + 1, b'_{25} = b_{25} - 1, \\ 
	b'_{33} = b_{33} + 1,  b'_{36} = b_{36} - 1, b'_{45} = b_{45} + 1, b'_{48} = b_{48} - 1, \\ 
	b'_{56} = b_{56} + 1, b'_{59} = b_{59} - 1 \ \text{if} \ (F4) \vspace{1pt}\\ 
	b'_{11} = b_{11} + 1,b'_{13} = b_{13} - 1, b'_{22} = b_{22} + 1, b'_{24} = b_{24} - 1, \\ 
	b'_{33} = b_{33} + 1,  b'_{35} = b_{35} - 1, b'_{44} = b_{44} + 1, b'_{48} = b_{48} - 1, \\ 
	b'_{55} = b_{55} + 1, b'_{59} = b_{59} - 1 \ \text{if} \ (F5) 
	\end{cases}
	\\
k=1 &: b'_{11} =b_{11} -1, b'_{12} =b_{12} +1, b'_{58} =b_{58} -1, b'_{59} =b_{59} +1 \\
k=2 &: 
	\begin{cases} 
	b'_{12} =b_{12} -1, b'_{13} =b_{13} +1, b'_{47} =b_{47} -1, b'_{48} =b_{48} +1 \ \text{if} \ b_{12} > b_{23}\\ 
	b'_{22} =b_{22} -1, b'_{23} =b_{23} +1, b'_{57} =b_{57} -1, b'_{58} =b_{58} +1 \ \text{if} \ b_{12} \leq b_{23} 
	\end{cases}
	\\
k=3 &: 
	\begin{cases} 
	b'_{13} =b_{13} -1, b'_{14} =b_{14} +1, b'_{36} =b_{36} -1, b'_{37} =b_{37} +1  \\ 
	\hspace{1cm} \text{if} \ b_{13}> b_{24}, b_{13}+b_{23} > b_{24}+b_{34} \\ 
	b'_{23} =b_{23} -1, b'_{24} =b_{24} +1, b'_{46} =b_{46} -1, b'_{47} =b_{47} +1  \\ 
	\hspace{1cm} \text{if} \ b_{13} \leq  b_{24}, b_{23} > b_{34}\\ 
	b'_{33} =b_{33} -1, b'_{34} =b_{34} +1, b'_{56} =b_{56} -1, b'_{57} =b_{57} +1  \\ 
	\hspace{1cm} \text{if} \ b_{23} \leq  b_{34}, b_{13}+b_{23} \leq b_{24}+b_{34} 
	\end{cases}
	\\
k=4 &: 
	\begin{cases} 
	b'_{14} =b_{14} -1, b'_{15} =b_{15} +1, b'_{25} =b_{25} -1, b'_{26} =b_{26} +1 \\ 
	\hspace{1cm} \text{if} \ b_{14}+b_{33}+b_{34} > b_{22}+b_{23}\\ 
	b'_{34} =b_{34} -1, b'_{35} =b_{35} +1, b'_{45} =b_{45} -1, b'_{46} =b_{46} +1 \\ 
	\hspace{1cm} \text{if} \ b_{14}+b_{33}+b_{34} \leq b_{22}+b_{23} 
	\end{cases}
	\\
k=5 &: 
	\begin{cases} 
	b'_{24} =b_{24} -1, b'_{25} =b_{25} +1, b'_{35} =b_{35} -1, b'_{36} =b_{36} +1 \\ 
	\hspace{1cm}  \text{if} \ b_{24}+b_{44} > b_{33}\\ 
	b'_{44} =b_{44} -1, b'_{45} =b_{45} +1, b'_{55} =b_{55} -1, b'_{56} =b_{56} +1 \\ 
	\hspace{1cm}  \text{if} \ b_{24}+b_{44} \leq b_{33}
	\end{cases}
\end{align*} and $b'_{ij} = b_{ij}$ otherwise.\\

For $b \in B^{5,l}$ if $\tilde{e}_k(b)$ or $\tilde{f}_k(b)$ does not belong to $B^{5,l}$, then we assume it to be $0$. The maps 
$\veps_k(b), \ \vphi_k(b)$ and $\text{wt}_k(b)$ for $k=0,1,2,3,4,5$ are given as follows. We observe that 
$\text{wt}_k(b) = \vphi_k(b) - \veps_k(b)$,
$\vphi(b) = \sum_{k=0}^5\vphi_k(b)\L_k$, $\veps(b) = \sum_{k=0}^5\veps_k(b)\L_k$ and $\text{wt}(b) = \vphi(b) - \veps(b)$.
\begin{align*}
\veps_0(b) 	&= \begin{cases} l+\text{max} \{-b_{45}-b_{46}-b_{47}-b_{48}, -b_{13}-b_{34}-b_{35}-b_{36}-b_{37}, \\
				\hspace{15pt} -b_{24}-b_{34}-b_{35}-b_{36}-b_{37}, -b_{13}-b_{14}-b_{23}-b_{24}-b_{25}-b_{26}, \\
				\hspace{15pt} -b_{13}-b_{23}-b_{35}-b_{36}-b_{37} \} \ \text{if} \ b \in B^{5,l},\\ 
				\text{max} \{-b_{45}-b_{46}-b_{47}-b_{48}, -b_{13}-b_{34}-b_{35}-b_{36}-b_{37}, \\
				\hspace{15pt} -b_{24}-b_{34}-b_{35}-b_{36}-b_{37}, -b_{13}-b_{14}-b_{23}-b_{24}-b_{25}-b_{26}, \\
				\hspace{15pt} -b_{13}-b_{23}-b_{35}-b_{36}-b_{37} \} \ \text{if} \ b \in B^{5,\infty}, \end{cases} \\
\veps_1(b) 	&= b_{12},\\
\veps_2(b) 	&= \text{max} \{b_{13}, -b_{12}+b_{13}+b_{23}\},\\
\veps_3(b) 	&= \text{max} \{b_{14}, -b_{13}+b_{14}+b_{24}, -b_{13}+b_{14}-b_{23}+b_{24}+b_{34}\},\\
\veps_4(b) 	&= \begin{cases} l+\text{max} \{-b_{11}-b_{12}-b_{13}-b_{14}, \\
				\hspace{15pt} -b_{11}-b_{12}-b_{13}-2b_{14}+b_{22}+b_{23}-b_{33}-b_{34}\} \ \text{if} \ b \in B^{5,l},\\ 
				\text{max} \{-b_{11}-b_{12}-b_{13}-b_{14}, \\
				\hspace{15pt} -b_{11}-b_{12}-b_{13}-2b_{14}+b_{22}+b_{23}-b_{33}-b_{34}\} \ \text{if} \ b \in B^{5,\infty}, \end{cases} \\
\veps_5(b) 	&= \text{max} \{b_{11}+b_{12}+b_{13}-b_{22}-b_{23}-b_{24},\\
			&\hspace{15pt} b_{11}+b_{12}+b_{13}-b_{22}-b_{23}-2b_{24}+b_{33}-b_{44}\},\\
\vphi_0(b) 	&= \begin{cases} l+\text{max} \{ -b_{11}-b_{12}-b_{13}-b_{14}, -b_{11}-b_{12}-b_{22}+b_{44}, \\
				\hspace{15pt} -b_{11}-b_{12}-b_{13}-b_{22}+b_{33}, -b_{11}-b_{12}-b_{22}-b_{24}+b_{33}, \\
				\hspace{15pt} -b_{11}-b_{12}-b_{13}-b_{22}-b_{23}+b_{33}+b_{34} \} \ \text{if} \ b \in B^{5,l},\\ 
				\text{max} \{ -b_{11}-b_{12}-b_{13}-b_{14}, -b_{11}-b_{12}-b_{22}+b_{44}, \\
				\hspace{15pt} -b_{11}-b_{12}-b_{13}-b_{22}+b_{33}, -b_{11}-b_{12}-b_{22}-b_{24}+b_{33}, \\
				\hspace{15pt} -b_{11}-b_{12}-b_{13}-b_{22}-b_{23}+b_{33}+b_{34} \}  \ \text{if} \ b \in B^{5,\infty}, \end{cases} \\
\vphi_1(b) 	&= b_{11}-b_{22},\\
\vphi_2(b) 	&= \text{max} \{b_{22}-b_{33}, b_{12}+b_{22}-b_{23}-b_{33}\},\\
\vphi_3(b) 	&= \text{max} \{b_{33}-b_{44}, b_{23}+b_{33}-b_{34}-b_{44}, b_{13}+b_{23}-b_{24}+b_{33}-b_{34}-b_{44}\},\\
\vphi_4(b) 	&= \begin{cases} l+\text{max} \{-b_{33}-b_{35}-b_{36}-b_{37}, \\
				\hspace{15pt} b_{14}-b_{22}-b_{23}+b_{34}-b_{35}-b_{36}-b_{37}\} \ \text{if} \ b \in B^{5,l},\\ 
				\text{max}  \{-b_{33}-b_{35}-b_{36}-b_{37}, \\
				\hspace{15pt} b_{14}-b_{22}-b_{23}+b_{34}-b_{35}-b_{36}-b_{37}\}  \ \text{if} \ b \in B^{5,\infty}, \end{cases} \\
\vphi_5(b) 	&= \text{max} \{b_{44}, b_{24}-b_{33}+2b_{44}\},\\
\text{wt}_0(b) 	&= -b_{11}-b_{12} +b_{23}+b_{24}+b_{25}+b_{26},\\
\text{wt}_1(b) 	&= b_{11}-b_{12}-b_{22},\\
\text{wt}_2(b) 	&= b_{12} -b_{13}+b_{22}-b_{23}-b_{33},\\
\text{wt}_3(b) 	&= b_{13}-b_{14}+b_{23}-b_{24}+b_{33}-b_{34}-b_{44},\\
\text{wt}_4(b)	&= b_{11}+b_{12} +b_{13}+2b_{14}-b_{22}-b_{23}+b_{34}-b_{35} -b_{36}-b_{37}\\
\text{wt}_5(b) 	&= -b_{11}-b_{12} -b_{13}+b_{22}+b_{23}+2b_{24}-b_{33}+2b_{44}.
\end{align*}

\noindent Choose elements $b^0_{\bf{0}}, b^0_{\bf{1}}, b^0_{\bf{2}}, b^0_{\bf{3}}, b^0_{\bf{4}}, b^0_{\bf{5}}$ where
\begin{align*}
(b^0_{\bf{0}})_{ij} &=1 &\text{if} \ (i,j) &= (1,5),(2,6),(3,7),(4,8),(5,9),  \\
(b^0_{\bf{1}})_{ij} &=1 &\text{if} \ (i,j) &= (1,1),(2,5),(3,6),(4,7),(5,8),  \\
(b^0_{\bf{2}})_{ij} &=1 &\text{if} \ (i,j) &= (1,1),(1,4),(2,2),(2,5),(3,5),(3,7),(4,6),(4,8),(5,7),(5,9),  \\
(b^0_{\bf{3}})_{ij} &=1 &\text{if} \ (i,j) &= (1,1),(1,3),(2,2),(2,4),(3,3),(3,5),(4,5),(4,8),(5,6),(5,9),  \\
(b^0_{\bf{4}})_{ij} &=1 &\text{if} \ (i,j) &= (1,2),(2,3),(3,4),(4,5),(5,9),  \\
(b^0_{\bf{5}})_{ij} &=1 &\text{if} \ (i,j) &= (1,1),(2,2),(3,3),(4,4),(5,5),  
\end{align*}
and $(b^0_{\bf{k}})_{ij} =0$ \ otherwise,\ for $0 \leq k \leq 5$.\\

As shown in \cite{KMN2}, the crystal $B^{5,l}$ is a perfect crystal with the set of minimal elements:
\begin{align*}
(B^{5,l})_\text{min}& = \{b \in B^{5,l} \mid \langle {\bf c} , \veps (b)\rangle = l\}\\
& =\left \{ \sum_{k=0}^5a_k b^0_{\bf k} \mid  a_k \in \mZ_{\geq 0},\ a_0+a_1+2a_2+2a_3+a_4+a_5=l \right\}.
\end{align*}

For $\lambda \in P_{cl}$, consider the crystal  $T_{\lambda} = \{t_\lambda\}$ with
\begin{align*}
&\tilde{e_k} (t_\lambda) = \tilde{f_k} (t_\lambda) =0,	&\veps_k (t_\lambda) = \vphi_k (t_\lambda) =-\infty,  \\ 
&\text{wt}(t_\lambda)=\lambda,
\end{align*}
$k=0,1,2,3,4,5$. Then for $\lambda, \mu \in P_{cl}$, $T_{\lambda}\otimes B^{5,l} \otimes T_{\mu}$ is a crystal with the structure given by 
\begin{align*}
\tilde{e_k}(t_{\lambda}\otimes b \otimes t_{\mu})&= t_{\lambda}\otimes \tilde{e_k}b \otimes t_{\mu}, &\tilde{f_k}(t_{\lambda}\otimes b \otimes t_{\mu})&= t_{\lambda}\otimes \tilde{f_k}b \otimes t_{\mu}, \\
\veps_k(t_{\lambda}\otimes b \otimes t_{\mu})&= \veps_k(b) - \langle \check\alpha_k,\lambda \rangle, &\vphi_k(t_{\lambda}\otimes b \otimes t_{\mu})&= \vphi_k(b) + \langle \check\alpha_k,\mu \rangle, \\
\text{wt}(t_{\lambda}\otimes b \otimes t_{\mu})&= \lambda+\mu+\text{wt}(b)
\end{align*}
where $t_{\lambda}\otimes b \otimes t_{\mu} \in T_{\lambda}\otimes B^{5,l} \otimes T_{\mu}$.

The notion of a coherent family of perfect crystals and its limit is defined in \cite{KKM}. In the following theorem we prove that the family of $D_5^{(1)}$ crystals $\Bfamfive$ form a coherent family with limit $B^{5, \infty}$ containing the special vector $b^{\infty} = {\bf{0}}$ (i.e. $(b^{\infty})_{ij} = 0$ for $i \leq j \leq i+4,\ 1 \leq i \leq 5$). 
%
\begin{theorem} The family of the perfect crystal $\Bfamfive$ forms a coherent family and the crystal $B^{5,\infty}$ is its limit with the vector $b^{\infty}$.
\end{theorem}
\begin{proof} Set $J=\{(l,b)| l \in \mZ_{>0}, b \in (B^{5,l})_\text{min}\}$. By (\cite{KKM}, Definition 4.1), we need to show that
\begin{enumerate}
\item wt$(b^{\infty})=0, \veps(b^{\infty})=\vphi(b^{\infty})=0,$
\item for any $(l,b) \in J$, there exists an embedding of crystals
$$f_{(l,b)} : T_{\veps(b)} \otimes B^{5, l} \otimes T_{-\vphi(b)} \longrightarrow B^{5, \infty}$$
where $f_{(l,b)}(t_{\veps(b)} \otimes b \otimes t_{-\vphi(b)}) = b^{\infty}$, and
\item $B^{5, \infty} = \cup_{(l,b) \in J}$ Im $f_{(l,b)}$.\\
\end{enumerate} 

Since $\veps_k(b^{\infty}) = 0, \vphi_k(b^{\infty}) = 0, 0 \leq k \leq 5$, we have $\veps(b^{\infty}) = 0, \vphi(b^{\infty}) = 0$ and hence wt$(b^{\infty}) = 0$ which proves $(1)$.\\

Let $l \in \mZ_{>0}$ and $b^0 = (b^0_{ij})$ be an element of $(B^{5,l})_{\text{min}}$. Then there exist $a_k \in \mZ_{\geq 0}, 0 \leq k \leq 5$ such that $a_0+a_1+2a_2+2a_3+a_4+a_5=l$ and 
\begin{align*}
b^0_{11} &= a_1+a_2+a_3+a_5,  \ b^0_{12} =a_4, \ b^0_{13}  = a_3, \ b^0_{14} =a_2, \ b^0_{15}=a_0, \\
b^0_{22} &= a_2+a_3+a_5, \ b^0_{23} =a_4, \ b^0_{24} = a_3, \ b^0_{25} = a_1+a_2, \ b^0_{26} = a_0, \ b^0_{33} = a_3+a_5,\\
b^0_{34} &=a_4, \ b^0_{35}=a_2+a_3, \ b^0_{36} =a_1, \ b^0_{37} = a_0+a_2, \ b^0_{44} = a_5, \ b^0_{45} = a_3+a_4, \\ 
b^0_{46} &= a_2,\ b^0_{47} =a_1, \ b^0_{48} =a_0+a_2+a_3, \ b^0_{55} =a_5, \ b^0_{56} = a_3, \ b^0_{57} = a_2, \ b^0_{58} = a_1, \\ 
b^0_{59} &= a_0+a_2+a_3+a_4, \ {\text{and}}\\
\vphi(b^0)&= a_0\Lambda_0+a_1\Lambda_1+a_2\Lambda_2+a_3\Lambda_3+a_4\Lambda_4+a_5\Lambda_5,\\
\veps(b^0)& = a_5\Lambda_0+a_4\Lambda_1+a_3\Lambda_2+a_2\Lambda_3+a_0\Lambda_4+a_1\Lambda_5. 
\end{align*}

For any $b= (b_{ij}) \in B^{5,l}$ we define a map
$$f_{(l,b^0)} : T_{\veps(b^0)} \otimes B^{5,l} \otimes T_{-\vphi(b^0)} \longrightarrow B^{5,\infty}$$
by $f_{(l,b^0)}( t_{\veps(b^0)} \otimes b \otimes t_{-\vphi(b^0)} )= b' =(b'_{ij}) $
where $b'_{ij} = b_{ij} - b^0_{ij}$ for all $i \leq j \leq i+4,\ 1 \leq i \leq 5$.
Then it is easy to see that
\begin{align*}
\veps_k(b') &= \begin{cases}
\veps_k(b) - a_{5-k} = \veps_k(b) - \langle \check\alpha_k, \veps(b^0)\rangle, \text{for} \; k \neq 4,5\\
\veps_4(b) - a_0 = \veps_4(b) - \langle \check\alpha_4, \veps(b^0)\rangle,  \text{for} \; k = 4\\
\veps_5(b) - a_1 = \veps_5(b) - \langle \check\alpha_5, \veps(b^0)\rangle,  \text{for} \; k = 5\\
\end{cases}\\
\vphi_k(b') &= \vphi_k(b) - a_k = \vphi_k(b) + \langle \check\alpha_k, -\vphi(b^0)\rangle, \text{for} \; 0 \leq k  \leq5.
\end{align*}
Hence we have
\begin{align*}
\veps_k(b') &= \veps_k(b) - \langle \check\alpha_k, \veps(b^0)\rangle = \veps_k(t_{\veps(b^0)}\otimes b\otimes t_{-\vphi(b^0)}), \\
\vphi_k(b') &= \vphi_k(b) + \langle \check\alpha_k, -\vphi(b^0)\rangle=\vphi_k(t_{\veps(b^0)}\otimes b\otimes t_{-\vphi(b^0)}),\\
\text{wt}(b') &= \sum_{k=0}^5(\vphi_k(b') - \veps_k(b'))\L_k = \text{wt}(b) +  \sum_{k=0}^5 \langle \check\alpha_k, -\vphi(b^0)\rangle\L_k + \sum_{k=0}^5 \langle \check\alpha_k, \veps(b^0)\rangle\L_k \\
&= \text{wt}(b) + \veps(b^0) - \vphi(b^0) = \text{wt}(t_{\veps(b^0)}\otimes b\otimes t_{-\vphi(b^0)}).
\end{align*}

For $0 \leq k \leq 5, b \in B^{5,l}$, it can be checked easily that the conditions for the action of $\e_k$ on $b' = b - b^0$ hold if and only if the conditions for the action of $\e_k$ on $b$ hold. Hence from the defined action of $\e_k$, we see that $\e_k(b') = \e_k(b) - b^0, 0\leq k \leq5$. This implies that 
$$
f_{(l,b^0)} (\tilde{e_k}( t_{\veps(b^0)} \otimes b \otimes t_{-\vphi(b^0)})) = f_{(l,b^0)} ( t_{\veps(b^0)} \otimes \e_k(b) \otimes t_{-\vphi(b^0)})
$$
$$
= \e_k(b) - b^0 = \e_k(b') = \e_k ( f_{(l, b^0)}( t_{\veps(b^0)} \otimes b \otimes t_{-\vphi(b^0)})).
$$
Similarly, we have  $f_{(l,b^0)} (\tilde{f_k}( t_{\veps(b^0)} \otimes b \otimes t_{-\vphi(b^0)})) =  \f_k ( f_{(l, b^0)}( t_{\veps(b^0)} \otimes b \otimes t_{-\vphi(b^0)}))$. Clearly the map $f_{(l,b^0)}$ is injective with $f_{(l,b^0)} ( t_{\veps(b^0)}  \otimes b^0 \otimes t_{-\vphi(b^0)})=b^{\infty}$. This proves (2). \\

We observe that $ \sum_{j=i}^{i+4} b'_{ij} =  \sum_{j=i}^{i+4} b_{ij} -  \sum_{j=i}^{i+4} b^0_{ij} = l-l = 0$ for all $1 \leq i \leq 5$. Also,
\begin{align*}
b'_{11} &= b_{11} - b^0_{11} = b_{55} + b_{56} + b_{57} + b_{58} - a_1 - a_2 - a_3 -a_5 \\
	&= b'_{55} + b'_{56} + b'_{57} + b'_{58},\\
b'_{11} + b'_{12} &= b_{11} - b^0_{11} + b_{12} - b^0_{12} \\
	&= b_{44} + b_{45} + b_{46} + b_{47} - a_1 - a_2 - a_3 - a_4 - a_5 \\
	&= b'_{44} + b'_{45} + b'_{46} + b'_{47},\\
b'_{11} + b'_{12} + b'_{13} &= b_{11} - b^0_{11} + b_{12} - b^0_{12} + b_{13}  - b^0_{13} \\
	&= b_{33} + b_{34} + b_{35} + b_{36} - a_1 - a_2 - 2a_3 - a_4 - a_5 \\
	&= b'_{33} + b'_{34} + b'_{35} + b'_{36},\\
b'_{11} + b'_{12} + b'_{13} + b'_{14} &= b_{11} - b^0_{11}  + b_{12} - b^0_{12} + b_{13} - b^0_{13} + b_{14} - b^0_{14} \\
	&= b_{22} + b_{23} + b_{24} + b_{25} - a_1 - 2a_2 - 2a_3 - a_4 - a_5 \\
	&= b'_{22} + b'_{23} + b'_{24} + b'_{25}.
\end{align*}
Similarly, we can show that $\sum_{j=i}^{5-t} b'_{ij} = \sum_{j=i+t}^{4+t} b'_{i+t,j},$\ for $2 \leq i \leq 4,\ 1 \leq t \leq 4$. Hence we have $B^{5,\infty} \supseteq \cup_{(l,b) \in J}$ Im $f_{(l,b)}$. To prove (3) we also need to show that $B^{5,\infty} \subseteq \cup_{(l,b) \in J}$ Im $f_{(l,b)}$. Let $b' = (b'_{ij}) \in B^{5,\infty}$. By (2), we can assume that $b' \neq b^{\infty}$. Set 
\begin{align*}
a_1 	&=   \text{max} \{- b'_{11} + b'_{22} ,- b'_{11}- b'_{12}+b'_{22}+b'_{23},  - b'_{11}- b'_{12}- b'_{13}+b'_{22}+b'_{23}+b'_{24}, 0\},\\
a_2 	&= \text{max} \{- b'_{22}+b'_{33} ,  - b'_{22} - b'_{23}+b'_{33}+b'_{34}, - b'_{14}, -b'_{25} - a_1, 0\}, \\
a_3 	&= \text{max} \{- b'_{33}+b'_{44} , - b'_{13}, - b'_{24}, -b'_{35} - a_2, 0\}, \\
a_4 	&= \text{max} \{- b'_{12}, - b'_{23}, - b'_{34}, -b'_{45} - a_3, 0\},\\
a_5 	&= \text{max} \{ -b'_{11} - a_1 - a_2 - a_3, -b'_{22} - a_2 - a_3, -b'_{33} - a_3, -b'_{44}, 0 \}, \\
a_0 	&= \text{max} \{ b'_{11} - a_2 - a_3 - a_4, b'_{11}+b'_{12} - a_2 - a_3, b'_{11}+b'_{12}+b'_{13} - a_2, \\
	&\hspace{40pt} b'_{11}+b'_{12}+b'_{13}+b'_{14}, 0 \}.
\end{align*}
Let $l=a_0+a_1+2a_2+2a_3+a_4+a_5$. Let $b^0 = (b^0_{ij})$ where 
\begin{align*}
b^0_{11} &= a_1+a_2+a_3+a_5,  \ b^0_{12} =a_4, \ b^0_{13}  = a_3, \ b^0_{14} =a_2, \ b^0_{15}=a_0, \\
b^0_{22} &= a_2+a_3+a_5, \ b^0_{23} =a_4, \ b^0_{24} = a_3, \ b^0_{25} = a_1+a_2, \ b^0_{26} = a_0, \ b^0_{33} = a_3+a_5,\\
b^0_{34} &=a_4, \ b^0_{35}=a_2+a_3, \ b^0_{36} =a_1, \ b^0_{37} = a_0+a_2, \ b^0_{44} = a_5, \ b^0_{45} = a_3+a_4, \\ 
b^0_{46} &= a_2,\ b^0_{47} =a_1, \ b^0_{48} =a_0+a_2+a_3, \ b^0_{55} =a_5, \ b^0_{56} = a_3, \ b^0_{57} = a_2, \ b^0_{58} = a_1, \\ 
b^0_{59} &= a_0+a_2+a_3+a_4.
\end{align*}
Then $\vphi(b^0) = a_0\Lambda_0+a_1\Lambda_1+a_2\Lambda_2+a_3\Lambda_3+a_4\Lambda_4+a_5\Lambda_5$
and  $\veps(b^0) = a_5\Lambda_0+a_4\Lambda_1+a_3\Lambda_2+a_2\Lambda_3+a_0\Lambda_4+a_1\Lambda_5$.
It is easy to see that $b^0 \in (B^{5,l})_{\text{min}}$. 

Set $b = (b_{ij})$ where $b_{ij} = b'_{ij} +b^0_{ij}$. Then $\sum_{j=i}^{i+4} b_{ij} = \sum_{j=i}^{i+4} b'_{ij} + \sum_{j=i}^{i+4} b^0_{ij} = 0 + l = l,\ 1 \leq i \leq 5$ and we observe that 
\begin{align*}
b_{11} &= b'_{11} + b^0_{11} = b'_{11} + a_1+a_2+a_3+a_5 \geq 0, \ \text{since} \; a_5 \geq -b'_{11} - a_1 - a_2 - a_3,\\ 
b_{12} &= b'_{12} + b^0_{12} = b'_{12} + a_4 \geq 0, \ \text{since} \; a_4 \geq - b'_{12},\\ 
b_{13} &= b'_{13} + b^0_{13} = b'_{13} + a_3 \geq 0, \ \text{since} \; a_3 \geq - b'_{13},\\ 
b_{14} &= b'_{14} + b^0_{14} = b'_{14} + a_2 \geq 0, \ \text{since} \; a_2 \geq - b'_{14},\\ 
b_{15} &= b'_{15} + b^0_{15} = b'_{15} + a_0 = - b'_{11} - b'_{12} - b'_{13} - b'_{14} + a_0 \geq 0, \\
	& \hspace{4.5cm} \text{since} \; a_0 \geq b'_{11}+b'_{12}+b'_{13}+b'_{14}.
\end{align*}
Similarly, we can show that $b_{ij} \in \mZ_{\geq 0}$ for $i \leq j \leq i+4,\ 2 \leq i \leq 5$. We also have
\begin{align*}
b_{33} &= b'_{33} + b^0_{33} = b'_{55} + b'_{56} + a_3  + a_5 = b_{55} + b_{56},\\
b_{33} + b_{34} &= b'_{33} + b^0_{33} + b'_{34} + b^0_{34} = b'_{44} + b'_{45} +  a_3 + a_4 + a_5 = b_{44} + b_{45},\\
b_{44} &= b'_{44} + b^0_{44} = b'_{55} + a_5 = b_{55}.
\end{align*}
Similarly, we see that $\sum_{j=i}^{5-t} b_{ij} = \sum_{j=i+t}^{4+t} b_{i+t,j},\ 1 \leq i \leq 2, 1 \leq t \leq 4$. \ Also,
\begin{align*}
b_{11} &= b'_{11} + b^0_{11} \geq b'_{22} + b^0_{22} = b_{22},\ \text{since} \; b^0_{11} - b^0_{22} = a_1 \geq - b'_{11} + b'_{22},\\ 
b_{22} &= b'_{22} + b^0_{22} \geq b'_{33} + b^0_{33} = b_{33},\ \text{since} \; b^0_{22} - b^0_{33} = a_2 \geq - b'_{22} + b'_{33},\\
b_{33} &= b'_{33} + b^0_{33} \geq b'_{44} + b^0_{44} = b_{44},\ \text{since} \; b^0_{33} - b^0_{44} = a_3 \geq - b'_{33} + b'_{44} .
\end{align*}
Similarly, $\sum_{j=i}^{t} b_{ij} \geq \sum_{j=i+1}^{t+1} b_{i+1,j}, 1 \leq i < t \leq 4$. Hence $b \in B^{5,l}$.

Then $f_{(l,b^0)}(t_{\veps(b^0)} \otimes b \otimes t_{-\vphi(b^0)}) = b',$ and 
$b' \in \cup_{(l,b) \in J}$ Im $f_{(l,b)}$ which proves (3).
\end{proof}

\section{Ultra-discretization of \bf{$\mathcal{V}(D_5^{(1)})$}} 
It is known that the ultra-discretization of a positive geometric crystal is a Kashiwara crystal \cite{BK, N}. In this section we apply the ultra-discretization functor $\mathcal{UD}$ to the positive geometric crystal $\mathcal{V}=\mathcal{V}(D_5^{(1)})$ constructed in  Section 2. Then we show that as crystal it is isomorphic to the crystal $B^{5, \infty}$ given in the last section which proves the conjecture in \cite{KNO} for this case.
As a set $\mathcal{X}=\mathcal{UD}(\mathcal{V}) = \mZ^{10}$. We denote the variables $x_m^{(l)}$ in $\cV$ by the same notation $x_m^{(l)}$ in 
$\mathcal{UD}(\mathcal{V}) = \mathcal{X}$.\\

Let $x=(x_4^{(2)},x_3^{(3)}, x_2^{(2)}, x_5^{(2)}, x_3^{(2)}, x_4^{(1)}, x_1^{(1)}, x_2^{(1)},x_3^{(1)}, x_5^{(1)}) \in \mathcal{X}$. By applying the ultra-discretization functor $\mathcal{UD}$ to the positive geometric crystal $\cV$ in Section 2, we have for $0 \leq k \leq 5$:

\begin{align*}
\mathcal{UD}(\gamma_k)(x) &=
\begin{cases}
-x_2^{(2)}-x_2^{(1)}, &k=0,\\
2x_1^{(1)}-x_2^{(2)}-x_2^{(1)}, &k=1,\\
-x_1^{(1)}+2x_2^{(2)} -x_3^{(3)}+2x_2^{(1)}-x_3^{(2)}-x_3^{(1)}, &k=2,\\
-x_2^{(2)}+2x_3^{(3)}-x_4^{(2)}-x_2^{(1)}+2x_3^{(2)}-x_5^{(2)}+2x_3^{(1)}-x_4^{(1)}-x_5^{(1)}, &k=3,\\
-x_3^{(3)}+2x_4^{(2)} -x_3^{(2)}-x_3^{(1)}+2x_4^{(1)}, &k=4,\\
-x_3^{(3)}-x_3^{(2)}+2x_5^{(2)}-x_3^{(1)}+2x_5^{(1)}, &k=5.
\end{cases} \\
\mathcal{UD}(\veps_k)(x) &= 
\begin{cases}
\text{max} \{x_5^{(1)}, x_2^{(2)}-x_3^{(3)}+x_3^{(1)}, x_3^{(2)}-x_5^{(2)}+x_3^{(1)}, x_2^{(2)}-x_4^{(2)}+x_2^{(1)}, &\\
\hspace{15pt}x_2^{(2)}-x_3^{(3)}+x_2^{(1)}-x_3^{(2)}+x_4^{(1)}\}, &k=0,\\
-x_1^{(1)}+x_2^{(2)}, &k=1,\\
\text{max} \{-x_2^{(2)}+x_3^{(3)}, x_1^{(1)}-2x_2^{(2)}+x_3^{(3)}-x_2^{(1)}+x_3^{(2)}\}, &k=2,\\
\text{max} \{-x_3^{(3)}+x_4^{(2)}, x_2^{(2)}-2x_3^{(3)}+x_4^{(2)}-x_3^{(2)}+x_5^{(2)}, &\\
\hspace{15pt}x_2^{(2)}-2x_3^{(3)}+x_4^{(2)}+x_2^{(1)}-2x_3^{(2)}+x_5^{(2)}-x_3^{(1)}+x_4^{(1)}\}, &k=3,\\
\text{max} \{-x_4^{(2)}, x_3^{(3)}+x_3^{(2)}-2x_4^{(2)}-x_4^{(1)}\}, &k=4,\\
\text{max} \{x_3^{(3)}-x_5^{(2)}, x_3^{(3)}+x_3^{(2)}-2x_5^{(2)}+x_3^{(1)}-x_5^{(1)}\}, &k=5.
\end{cases}
\end{align*}
We define
\begin{align*}
\breve{c_2} &=\text{max}\{c+x_2^{(2)}+x_2^{(1)}, x_3^{(2)}+x_3^{(1)}\}-\text{max}\{x_2^{(2)}+x_2^{(1)}, x_3^{(2)}+x_3^{(1)}\},\\
\breve{c_{3_1}} &=\text{max}\{c+x_3^{(3)}+2x_3^{(2)}+x_3^{(1)}, x_2^{(2)}+x_5^{(2)}+x_3^{(2)}+x_3^{(1)},x_2^{(2)}+x_5^{(2)}+x_4^{(1)}+x_2^{(1)}\}-\\
			& \hspace{15pt} \text{max}\{x_3^{(3)}+2x_3^{(2)}+x_3^{(1)}, x_2^{(2)}+x_5^{(2)}+x_3^{(2)}+x_3^{(1)},x_2^{(2)}+x_5^{(2)}+x_4^{(1)}+x_2^{(1)}\},\\
\breve{c_{3_2}} &=\text{max}\{c+x_3^{(3)}+2x_3^{(2)}+x_3^{(1)}, c+x_2^{(2)}+x_5^{(2)}+x_3^{(2)}+x_3^{(1)},x_2^{(2)}+x_5^{(2)}+x_4^{(1)}+x_2^{(1)}\}-\\
			& \hspace{15pt} \text{max}\{c+x_3^{(3)}+2x_3^{(2)}+x_3^{(1)}, x_2^{(2)}+x_5^{(2)}+x_3^{(2)}+x_3^{(1)},x_2^{(2)}+x_5^{(2)}+x_4^{(1)}+x_2^{(1)}\},\\
\breve{c_4} &=\text{max}\{c+x_4^{(2)}+x_4^{(1)}, x_3^{(3)}+x_3^{(2)}\}-\text{max}\{x_4^{(2)}+x_4^{(1)}, x_3^{(3)}+x_3^{(2)}\},\\
\breve{c_5} &=\text{max}\{c+x_5^{(2)}+x_5^{(1)}, x_3^{(2)}+x_3^{(1)}\}-\text{max}\{x_5^{(2)}+x_5^{(1)}, x_3^{(2)}+x_3^{(1)}\},\\
\breve{B} &=x_3^{(1)}+\text{max}\{x_2^{(2)}-x_3^{(3)},x_3^{(2)}- x_5^{(2)}\},\\
\breve{C} &=x_2^{(2)}+x_2^{(1)}+\text{max}\{-x_4^{(2)},x_4^{(1)}-x_3^{(3)}- x_3^{(2)}\},\;\;
\breve{A} = \breve{B}+\breve{C}.
\end{align*}
Then we have
\begin{align*}
\mathcal{UD}(e_k^{c})(x) &=
\begin{cases}
(x_4^{(2)'},x_3^{(3)'}, x_2^{(2)}-c, x_5^{(2)'}, x_3^{(2)'}, x_4^{(1)'}, x_1^{(1)}-c, x_2^{(1)}-c,x_3^{(1)'}, x_5^{(1)'}), &k=0,\\
(x_4^{(2)},x_3^{(3)}, x_2^{(2)}, x_5^{(2)}, x_3^{(2)}, x_4^{(1)}, x_1^{(1)}+c, x_2^{(1)},x_3^{(1)}, x_5^{(1)}), &k=1,\\
(x_4^{(2)},x_3^{(3)}, x_2^{(2)}+\breve{c_2}, x_5^{(2)}, x_3^{(2)}, x_4^{(1)}, x_1^{(1)}, x_2^{(1)}+c-\breve{c_2},x_3^{(1)}, x_5^{(1)}), &k=2,\\
(x_4^{(2)},x_3^{(3)}+\breve{c_{3_1}}, x_2^{(2)}, x_5^{(2)}, x_3^{(2)}+\breve{c_{3_2}}, x_4^{(1)}, x_1^{(1)}, x_2^{(1)},&\\
\hspace{15pt}x_3^{(1)}+c-\breve{c_{3_1}}-\breve{c_{3_2}}, x_5^{(1)}), &k=3,\\
(x_4^{(2)}+\breve{c_4},x_3^{(3)}, x_2^{(2)}, x_5^{(2)}, x_3^{(2)}, x_4^{(1)}+c-\breve{c_4}, x_1^{(1)}, x_2^{(1)},x_3^{(1)}, x_5^{(1)}), &k=4,\\
(x_4^{(2)},x_3^{(3)}, x_2^{(2)}, x_5^{(2)}+\breve{c_5}, x_3^{(2)}, x_4^{(1)}, x_1^{(1)}, x_2^{(1)},x_3^{(1)}, x_5^{(1)}+c-\breve{c_5}), &k=5,
\end{cases} \\
\text{where} &\\
x_3^{(1)'} &=x_3^{(1)} + \text{max}\{x_5^{(1)},\breve{A}\}-\text{max}\{c+x_5^{(1)}, c+\breve{B}, \breve{C}\},\\
x_3^{(2)'} &=-c +\text{max}\{x_2^{(2)}+x_5^{(2)}-x_3^{(3)}, x_3^{(2)}\}+\text{max}\{c+x_5^{(1)}, c+\breve{B}, \breve{C}\}-\\
		& \hspace{15pt} \text{max}\{c+x_5^{(1)}, c+\breve{B}, \breve{C},x_2^{(2)}+x_5^{(2)}-x_3^{(3)}- x_3^{(2)}+ \text{max}\{c+x_5^{(1)}, \breve{A}\} \},\\
x_3^{(3)'} &=-c+ x_3^{(3)}+ x_3^{(2)}- \text{max}\{x_5^{(1)}+\breve{A}\}-\text{max}\{x_2^{(2)}+x_5^{(2)}-x_3^{(3)}, x_3^{(2)}\} +\\
		& \hspace{15pt} \text{max}\{c+x_5^{(1)}, c+\breve{B}, \breve{C},x_2^{(2)}+x_5^{(2)}-x_3^{(3)}- x_3^{(2)}+ \text{max}\{c+x_5^{(1)}, \breve{A}\}\}  ,\\
x_4^{(1)'} &=-c+x_4^{(1)}+\text{max}\{x_3^{(3)}+x_3^{(2)},x_4^{(2)}+x_4^{(1)}\}-\\
		& \hspace{15pt} \text{max}\{x_4^{(2)}+x_4^{(1)}, -c+x_3^{(3)}+x_3^{(2)}- \text{max}\{x_5^{(1)}+\breve{A}\}+ \text{max}\{c+x_5^{(1)}, c+\breve{B}, \breve{C}\}\}  ,\\
x_4^{(2)'} &=x_4^{(2)}-\text{max}\{x_3^{(3)}+x_3^{(2)},x_4^{(2)}+x_4^{(1)}\}+\\
		& \hspace{15pt} \text{max}\{x_4^{(2)}+x_4^{(1)}, -c+x_3^{(3)}+x_3^{(2)}- \text{max}\{x_5^{(1)}+\breve{A}\}+ \text{max}\{c+x_5^{(1)}, c+\breve{B}, \breve{C}\}\}  ,\\
x_5^{(1)'} &=x_5^{(1)}+\text{max}\{x_5^{(1)},\breve{A}\}-\text{max}\{c+x_5^{(1)},\breve{A}\},\\
x_5^{(2)'} &=-c+x_5^{(2)}-\text{max}\{x_5^{(1)},\breve{A}\}+\text{max}\{c+x_5^{(1)},\breve{A}\}.
\end{align*}
As shown in \cite{BK, N}, $\mathcal{X}$ with maps $\e_k, \f_k :\mathcal{X} \longrightarrow \mathcal{X}\cup \{0\}, \; \veps_k, \vphi_k : \mathcal{X} \longrightarrow \mZ, \; 0\leq k \leq 5$ and $\text{wt}: \mathcal{X} \longrightarrow P_{cl}$ is a Kashiwara crystal where for $x \in \mathcal{X}$
\begin{align*}
\e_k(x) &= \mathcal{UD}(e_k^c)(x)\arrowvert_{c=1}, \; \f_k(x) = \mathcal{UD}(e_k^c)(x)\arrowvert_{c=-1}, \\
\text{wt}(x) &= \sum_{k=0}^5\text{wt}_k(x)\L_k, \text{where} \; \text{wt}_k(x) = \mathcal{UD}(\gamma_k)(x), \\
\veps_k(x) &= \mathcal{UD}(\veps_k)(x), \; \vphi_k(x) = \text{wt}_k(x) + \veps_k(x).
\end{align*}

In particular, the explicit actions of $\f_k, 1\leq k \leq 5$ on $\mathcal{X}$ is given as follows.
\begin{align*}
\tilde{f_1}(x) & =(x_4^{(2)},x_3^{(3)}, x_2^{(2)}, x_5^{(2)}, x_3^{(2)}, x_4^{(1)}, x_1^{(1)}-1, x_2^{(1)},x_3^{(1)}, x_5^{(1)}),\\
\tilde{f_2}(x)  & =\begin{cases} (x_4^{(2)},x_3^{(3)}, x_2^{(2)}-1, x_5^{(2)}, x_3^{(2)}, x_4^{(1)}, x_1^{(1)}, x_2^{(1)},x_3^{(1)}, x_5^{(1)}) \\
\hspace{15pt}\text{if} \  x_2^{(2)}+x_2^{(1)} > x_1^{(1)}+x_3^{(2)},\\ 
(x_4^{(2)},x_3^{(3)}, x_2^{(2)}, x_5^{(2)}, x_3^{(2)}, x_4^{(1)}, x_1^{(1)}, x_2^{(1)}-1,x_3^{(1)}, x_5^{(1)}) \\
\hspace{15pt}\text{if} \  x_2^{(2)}+x_2^{(1)} \leq x_1^{(1)}+x_3^{(2)}, \end{cases}\\
\tilde{f_3}(x)  & =\begin{cases} (x_4^{(2)},x_3^{(3)}-1, x_2^{(2)}, x_5^{(2)}, x_3^{(2)}, x_4^{(1)}, x_1^{(1)}, x_2^{(1)},x_3^{(1)}, x_5^{(1)}) \\ 
\hspace{15pt}\text{if} \ x_3^{(3)}+x_3^{(2)} > x_2^{(2)}+x_5^{(2)},\ x_3^{(3)}+2x_3^{(2)}+x_3^{(1)} > x_2^{(2)}+x_2^{(1)}+x_5^{(2)}+x_4^{(1)}, \\ (x_4^{(2)},x_3^{(3)}, x_2^{(2)}, x_5^{(2)}, x_3^{(2)}-1, x_4^{(1)}, x_1^{(1)}, x_2^{(1)},x_3^{(1)}, x_5^{(1)}) \\
\hspace{15pt} \text{if} \ x_3^{(3)}+x_3^{(2)} \leq x_2^{(2)}+x_5^{(2)}, \ x_3^{(2)}+x_3^{(1)} > x_2^{(1)}+x_4^{(1)},\\ 
(x_4^{(2)},x_3^{(3)}, x_2^{(2)}, x_5^{(2)}, x_3^{(2)}, x_4^{(1)}, x_1^{(1)}, x_2^{(1)},x_3^{(1)}-1, x_5^{(1)}) \\
\hspace{15pt} \text{if} \ x_3^{(2)}+x_3^{(1)} \leq x_2^{(1)}+x_4^{(1)},  \ x_3^{(3)}+2x_3^{(2)}+x_3^{(1)} \leq x_2^{(2)}+x_2^{(1)}+x_5^{(2)}+x_4^{(1)}, \end{cases}\\
\tilde{f_4}(x)  & =\begin{cases} (x_4^{(2)}-1,x_3^{(3)}, x_2^{(2)}, x_5^{(2)}, x_3^{(2)}, x_4^{(1)}, x_1^{(1)}, x_2^{(1)},x_3^{(1)}, x_5^{(1)}) \\
\hspace{15pt}\text{if} \  x_4^{(2)}+x_4^{(1)}> x_3^{(3)}+x_3^{(2)}\\ 
(x_4^{(2)},x_3^{(3)}, x_2^{(2)}, x_5^{(2)}, x_3^{(2)}, x_4^{(1)}-1, x_1^{(1)}, x_2^{(1)},x_3^{(1)}, x_5^{(1)}) \\
\hspace{15pt}\text{if} \  x_4^{(2)}+x_4^{(1)} \leq x_3^{(3)}+x_3^{(2)},  \end{cases}\\
\tilde{f_5}(x) & =\begin{cases} (x_4^{(2)},x_3^{(3)}, x_2^{(2)}, x_5^{(2)}-1, x_3^{(2)}, x_4^{(1)}, x_1^{(1)}, x_2^{(1)},x_3^{(1)}, x_5^{(1)}) \\
\hspace{15pt}\text{if} \  x_5^{(2)}+x_5^{(1)} > x_3^{(2)}+x_3^{(1)},\\ 
(x_4^{(2)},x_3^{(3)}, x_2^{(2)}, x_5^{(2)}, x_3^{(2)}, x_4^{(1)}, x_1^{(1)}, x_2^{(1)},x_3^{(1)}, x_5^{(1)}-1) \\
\hspace{15pt}\text{if} \  x_5^{(2)}+x_5^{(1)} \leq x_3^{(2)}+x_3^{(1)}. \end{cases}
\end{align*}
To determine the explicit action of $\tilde{f_0}(x)$ we define conditions $(\breve{F1})-(\breve{F5})$ as follows.\\

\begin{align*}
(\breve{F1}) 	& \hspace{5pt}x_2^{(2)}+x_2^{(1)} \geq x_4^{(2)}+x_5^{(1)}, x_3^{(3)}+x_3^{(2)} \geq x_4^{(2)}+x_4^{(1)}, \\
		& \hspace{15pt}x_3^{(3)}+x_2^{(1)} \geq x_4^{(2)}+x_3^{(1)}+(-x_2^{(2)}+x_3^{(3)}+ x_3^{(2)}-x_5^{(2)})_{+}, \\
(\breve{F2}) 	&\hspace{5pt}x_2^{(2)}+x_2^{(1)} +x_4^{(1)} \geq x_3^{(3)}+x_3^{(2)}+x_5^{(1)}, x_4^{(2)}+x_4^{(1)} > x_3^{(3)}+x_3^{(2)},\\
		& \hspace{15pt}x_2^{(1)}+x_4^{(1)} \geq x_3^{(2)}+x_3^{(1)}+(-x_2^{(2)}+x_3^{(3)}+ x_3^{(2)}-x_5^{(2)})_{+}, \\
(\breve{F3}) 	&\hspace{5pt}x_2^{(2)}+x_3^{(1)} \geq x_3^{(3)}+x_5^{(1)}, x_2^{(2)}+x_5^{(2)} \geq x_3^{(3)}+x_3^{(2)},\\
		& \hspace{15pt}x_4^{(2)}+x_3^{(1)} >x_3^{(3)}+x_2^{(1)}+(-x_3^{(3)}+x_4^{(2)}- x_3^{(2)}+x_4^{(1)})_{+}, \\
(\breve{F4}) 	&\hspace{5pt}x_3^{(2)}+x_3^{(1)} \geq x_5^{(2)}+x_5^{(1)}, x_3^{(3)}+x_3^{(2)} \geq x_2^{(2)}+x_5^{(2)},\\
		& \hspace{15pt}x_4^{(2)}+x_3^{(2)}+x_3^{(1)} > x_2^{(2)}+x_2^{(1)}+b_{24}+(-x_3^{(3)}+x_4^{(2)}- x_3^{(2)}+x_4^{(1)})_{+}, \\
(\breve{F5}) 	&\hspace{5pt}x_3^{(3)}+x_5^{(1)}> x_2^{(2)}+x_3^{(1)}+(-x_2^{(2)}+x_3^{(3)}+ x_3^{(2)}-x_5^{(2)})_{+}, \\
		& \hspace{15pt} x_4^{(2)}+x_5^{(1)}> x_2^{(2)}+x_2^{(1)}+(-x_3^{(3)}+x_4^{(2)}- x_3^{(2)}+x_4^{(1)})_{+}.\\
\end{align*}

Then for $x \in \mathcal{X}$ we have $\f_0(x) = \mathcal{UD}(e_0^c)(x)\arrowvert_{c=-1}$ given by

\begin{align*}
\tilde{f_0}(x) = \begin{cases} 
&(x_4^{(2)}+1,x_3^{(3)}+1, x_2^{(2)}+1, x_5^{(2)}+1, x_3^{(2)}+1, x_4^{(1)}, x_1^{(1)}+1, x_2^{(1)}+1,x_3^{(1)}, x_5^{(1)}) \\ & \hspace{15pt}\text{if} \ (\breve{F1}),  \vspace{1pt}\\ 
&(x_4^{(2)},x_3^{(3)}+1, x_2^{(2)}+1, x_5^{(2)}+1, x_3^{(2)}+1, x_4^{(1)}+1, x_1^{(1)}+1, x_2^{(1)}+1,x_3^{(1)}, x_5^{(1)}) \\ & \hspace{15pt}\text{if} \ (\breve{F2}),  \vspace{1pt}\\ 
&(x_4^{(2)},x_3^{(3)}+1, x_2^{(2)}+1, x_5^{(2)}+1, x_3^{(2)}, x_4^{(1)}+1, x_1^{(1)}+1, x_2^{(1)}+1,x_3^{(1)}+1, x_5^{(1)}) \\ & \hspace{15pt}\text{if} \ (\breve{F3}),  \vspace{1pt}\\ 
&(x_4^{(2)},x_3^{(3)}, x_2^{(2)}+1, x_5^{(2)}+1, x_3^{(2)}+1, x_4^{(1)}+1, x_1^{(1)}+1, x_2^{(1)}+1,x_3^{(1)}+1, x_5^{(1)}) \\ & \hspace{15pt}\text{if} \ (\breve{F4}),  \vspace{1pt}\\ 
&(x_4^{(2)},x_3^{(3)}, x_2^{(2)}+1, x_5^{(2)}, x_3^{(2)}+1, x_4^{(1)}+1, x_1^{(1)}+1, x_2^{(1)}+1,x_3^{(1)}+1, x_5^{(1)}+1) \\ & \hspace{15pt}\text{if} \ (\breve{F5}).   
 \end{cases}\\
\end{align*}

\begin{theorem} The map
\begin{displaymath}
\begin{array}{lccc}
\Omega : & B^{5,\infty} & \rightarrow & \mathcal{X},\\
&b=(b_{ij})_{i \leq j \leq i+4,\ 1 \leq i \leq 5} &\mapsto &x=(x_4^{(2)},x_3^{(3)}, x_2^{(2)}, x_5^{(2)}, x_3^{(2)}, x_4^{(1)}, x_1^{(1)}, x_2^{(1)},x_3^{(1)}, x_5^{(1)})
\end{array}
\end{displaymath}
defined by 
\begin{align*}
x_m^{(l)} &= \begin{cases}
\sum_{j=m-l+1}^m b_{m-l+1, j}, \ \ \text{for} \; \ \ m= 1, 2, 3\\
\sum_{j=m-2l+1}^m b_{m-2l+1, j}, \ \ \text{for} \; \ \ m= 4 \\
\sum_{j=m-2l+1}^{m-1} b_{m-2l+1, j}, \ \ \text{for} \; \ \ m= 5.
\end{cases}
\end{align*}
is an isomorphism of crystals.
\end{theorem}
\begin{proof} First we observe that the map $\Omega^{-1} : \mathcal{X} \rightarrow B^{5, \infty}$ is given by $\Omega^{-1}(x)=b$ where
\begin{align*}
b_{11} &= x_1^{(1)},  \ b_{12} =x_2^{(2)}-x_1^{(1)}, \ b_{13}  = x_3^{(3)}-x_2^{(2)}, \ b_{14} =x_4^{(2)}-x_3^{(3)}, \ b_{15}=-x_4^{(2)}, \\
b_{22} &= x_2^{(1)}, \ b_{23} =x_3^{(2)}-x_2^{(1)}, \ b_{24} = x_5^{(2)}-x_3^{(2)}, \ b_{25} = x_4^{(2)}-x_5^{(2)}, \ b_{26} = -x_4^{(2)}, \\ 
b_{33} &= x_3^{(1)}, \ b_{34} =x_4^{(1)}-x_3^{(1)}, \ b_{35}=x_5^{(2)}-x_4^{(1)}, \ b_{36} =x_3^{(3)}-x_5^{(2)}, \ b_{37} = -x_3^{(3)}, \\ 
b_{44} &= x_5^{(1)}, \ b_{45} = x_4^{(1)}-x_5^{(1)}, \ b_{46} = x_3^{(2)}-x_4^{(1)},\ b_{47} =x_2^{(2)}-x_3^{(2)}, \ b_{48} =-x_2^{(2)}, \\ 
b_{55} &= x_5^{(1)}, \ b_{56} = x_3^{(1)}-x_5^{(1)}, \ b_{57} = x_2^{(1)}-x_3^{(1)}, \ b_{58} = x_1^{(1)}-x_2^{(1)}, \ b_{59} = -x_1^{(1)}.
\end{align*}
Hence the map $\Omega$ is bijective.  To prove that $\Omega$ is an isomorphism of crystals we need to show 
that for $b \in B^{5,\infty}$ and $0 \leq k \leq 5$ we have:
\begin{align*}
\Omega(\tilde{f_k} (b)) 	&= \tilde{f_k} (\Omega(b)),\\
\Omega(\tilde{e_k} (b)) 	&= \tilde{e_k} (\Omega(b)),\\
\text{wt}_k (\Omega(b))	&= \text{wt}_k(b),\\
\veps_k (\Omega(b))		&= \veps_k(b).
\end{align*}
Hence $\vphi_k(\Omega(b)) = \text{wt}_k (\Omega(b)) + \veps_k (\Omega(b)) = \text{wt}_k(b) + \veps_k(b) = \vphi_k(b)$.	We observe that the conditions for the action of $\f_k$ on $\Omega(b)$ in $\mathcal{X}$ hold if and only if the corresponding conditions for the action of $\f_k$  on $b$ in $B^{5, \infty}$ hold for all $0\leq k \leq 5$. Suppose $\Omega(b) = x$ and $x_2^{(2)} + x_2^{(1)} > x_1^{(1)} + x_3^{(2)}$, then
$b_{11} + b_{12} + b_{22} > b_{11} + b_{22} + b_{23}$ and $\f_2(x) = (x_4^{(2)}, x_3^{(3)}, x_2^{(2)} - 1, x_5^{(2)}, x_3^{(2)}, x_4^{(1)},x_1^{(1)},x_2^{(1)},x_3^{(1)},x_5^{(1)}) = \Omega(\f_2(b))$. Similarly, we can show $\tilde{f_k} (\Omega(b)) = \tilde{f_k} (\Omega(b))$  and $\tilde{e_k} (\Omega(b)) = \tilde{e_k} (\Omega(b))$ for $k=0,1,3,4,5$. We also have
$\text{wt}_0(\Omega(b)) = \text{wt}_0(x) = -x_2^{(2)} - x_2^{(1)} = - b_{11} - b_{12} - b_{22} = - b_{11} - b_{12} + b_{23} + b_{24} + b_{25} + b_{26} = \text{wt}_0(b)$ for all $b \in B^{5,\infty}$. Similarly, $\text{wt}_k (\Omega(b)) = \text{wt}_k(b)$ for $1 \leq k \leq 5$. Also, $\veps_5 (\Omega(b)) = \veps_5(x) = \text{max} \{x_3^{(3)}-x_5^{(2)}, x_3^{(3)}+x_3^{(2)}-2x_5^{(2)}+x_3^{(1)}-x_5^{(1)}\} = \text{max} \{b_{11}+b_{12}+b_{13}-b_{22}-b_{23}-b_{24}, b_{11}+b_{12}+b_{13}-b_{22}-b_{23}-2b_{24}+b_{33}-b_{44}\}= \veps_5(b)$. Similarly, 
$\veps_k (\Omega(b)) = \veps_k(b)$ for $0 \leq k \leq 4$ which completes the proof.
\end{proof}

\bibliographystyle{amsalpha}

\end{document}